\documentclass[10pt]{article}
\usepackage{indentfirst}
\usepackage{comment}
\usepackage{amssymb}
\usepackage{amsmath}
\usepackage{amsthm}
\usepackage{mathrsfs}
\usepackage{xcolor}
\usepackage{bm}
\usepackage{stmaryrd}
\usepackage{epstopdf}
\usepackage{multirow}
\usepackage{booktabs}
\usepackage{float}
\usepackage{caption}
\usepackage{subfigure}
\usepackage{graphicx}
\usepackage{fancyhdr}
\usepackage[nocompress]{cite}
\usepackage{algorithm}
\usepackage{algorithmic}
\usepackage[colorlinks,linkcolor=blue,anchorcolor=red,citecolor=blue]{hyperref}
\numberwithin{equation}{section}
\usepackage{xcolor}

\newtheorem{theorem}{Theorem}[section]

\newtheorem{proposition}[theorem]{Proposition}
\newenvironment{keywords}{\paragraph{Key words.}}{}

\newcommand{\gc}{\mathrm{gc}}

\graphicspath{{figures/}}

\title{Fast and scalable multi-GPU sliced FFT solvers for asymptotically compatible discretizations of 3D nonlocal diffusion problems}
\newcommand{\shorttitle}{Multi-GPU solvers for 3D nonlocal diffusion problems}
\author{Yulong Li\thanks{School of Mathematics and Statistics, and National Center for Applied Mathematics in Hubei, Wuhan University, Wuhan 430072, China.}
\and
Jiale Meng\thanks{Beijing Computational Science Research Center, Beijing, 100193, China.}
\and
Zhenlin Guo\thanks{Co-corresponding author. Beijing Computational Science Research Center, Beijing, 100193, China (zguo@csrc.ac.cn).}
\and
Jiwei Zhang\thanks{Co-corresponding author. School of Mathematics and Statistics, and Hubei Key Laboratory of Computational Science, Wuhan University, Wuhan 430072, China (jiweizhang@whu.edu.cn).}
}
\date{}
\begin{document}
\maketitle

\begin{abstract}
Large-scale 3D peridynamic transient diffusion simulations are constrained by the computational cost and memory requirements of nonlocal interactions. 
We propose a fast and scalable multi-GPU sliced FFT algorithm based on the recently developed asymptotically compatible finite difference (ACFD) scheme. By exploiting the discrete operator's structure and finite interaction range, the algorithm combines FFT acceleration with local inter-GPU communication. It preserves the ACFD scheme while avoiding the global data redistribution required by full 3D  FFTs. We build a matrix-free GPU conjugate gradient (CG) solver, use multigrid preconditioning to reduce iteration counts, and introduce a human-audited configuration agent for efficient cross-level configuration. Strong-scaling tests show that sliced FFT--CG achieves 93.2\% parallel efficiency on eight GPUs, compared with 28.2\% for full FFT--CG on the same problem and hardware. The framework supports approximately 10.29 billion degrees of freedom on eight RTX~4090 GPUs with 24~GB of memory each. The insulated penny-shaped cutout benchmark demonstrates transient simulation with approximately 8.58 billion degrees of freedom. Three-dimensional  star-shaped interface evolution further demonstrates convergence toward the local solution in a nonlinear phase-field problem. These results demonstrate that sliced FFT provides a scalable multi-GPU approach for fast computation of 3D  peridynamic diffusion.
\end{abstract}

\begin{keywords}
Nonlocal diffusion; asymptotic compatibility; sliced FFT; multi-GPU parallel computing; multigrid preconditioning; agent-assisted optimization.
\end{keywords}

\section{Introduction}\label{sec:introduction}

Peridynamics (PD) formulates continuum interactions over finite distances through integral equations, providing a unified nonlocal description of problems with spatial discontinuities, such as material fracture~\cite{Silling2000Peridynamics,DuBook2019}. Originally developed for solid mechanics, this framework has been extended to heat and mass transport, with PD diffusion models replacing local diffusion operators by neighborhood integrals~\cite{BobaruDuangpanya2010Heat,Du2012NonlocalDiffusion,DElia2020NonlocalReview}. These models have been applied to transient heat conduction in domains containing cracks or cutouts and to mass transport during corrosion damage evolution~\cite{BobaruDuangpanya2012Heat,Jafarzadeh2021FCBM,ChenBobaru2015Corrosion}. However, evaluating a volume integral at every spatial node generates extensive neighborhood interactions, making 3D  simulations computationally expensive.

For a fixed interaction radius \(\delta\), refining the mesh size \(h\) increases the number of neighbors per node in 3D as \((\delta/h)^3\). On a fixed domain, direct matrix--vector multiplication and explicit matrix storage can therefore cost \(\mathcal O(N_{\mathrm{dof}}^2)\), where \(N_{\mathrm{dof}}\) is the total number of degrees of freedom~\cite{Jafarzadeh2021FCBM,VollmannSchulz2019}. Decreasing \(\delta\) together with \(h\) in a fixed ratio bounds the neighbor count, but convergence under mesh refinement cannot ensure the correctness of this joint limit; some conventional discretizations can even converge to a wrong local problem~\cite{TianDu2014AC}. To address this issue, Tian and Du established the framework of asymptotic compatibility (AC), requiring convergence both to the nonlocal solution as \(h\to0\) at fixed \(\delta\) and to the correct local solution as \(\delta\) and \(h\) vanish together~\cite{TianDu2014AC,TianDu2020Review}.

Within this framework, Du et al. developed AC finite difference (ACFD) schemes for multidimensional nonlocal diffusion~\cite{Du2019ACFD}. 
They provide a theoretical foundation for reliable engineering computation with nonlocal models, which also can resolve the correct local limits. 
 Subsequent studies further analyzed how interaction-neighborhood approximations affect asymptotic compatibility~\cite{DuXieYin2022Neighborhood,DuXieYinZhang2025Neighborhood}. However, the development of high-performance solvers for 3D  ACFD schemes remains unexplored. Translating their convergence advantages into large-scale computational capability is therefore a key step toward engineering applications. We adopt the recently developed second-order ACFD scheme based on an approximate-ball construction using clipping~\cite{Ding2026Approxcaps}, which yields a structured, symmetric positive definite discrete system on uniform grids.

Exploiting the massive thread parallelism of GPUs to accelerate neighborhood evaluation is an important route to improving the efficiency of peridynamic simulations. OpenCL implementations~\cite{Mossaiby2017OpenCL,Boys2021PeriPy}, CUDA parallelism at both particle and bond levels~\cite{Wang2022GPU}, and memory optimization~\cite{Bartlett2023Memory} have reduced runtimes and increased attainable problem sizes. For example, PeriPy achieves speedups of 1.4--10 over an existing OpenCL implementation in the reported tests~\cite{Boys2021PeriPy}. Wang et al. further simulate heat conduction with moving sources using approximately 76 million particles on three GPUs~\cite{Wang2024MGPB}. These advances primarily arise from parallel execution and storage optimization of neighborhood calculations, whose core remains the evaluation of particle or bond interactions. Thus, although GPUs execute these calculations faster, they still face substantial arithmetic work that grows with the number of neighbors per node. 

A complementary line of research exploits the Toeplitz and convolution structures of nonlocal operators, using FFTs to reduce the cost of operator evaluation to \(\mathcal O(N_{\mathrm{dof}}\log N_{\mathrm{dof}})\)~\cite{WangTian2014FastCollocation,VollmannSchulz2019,DuYang2017Fourier}. Jafarzadeh et al.~\cite{Jafarzadeh2021FCBM} combine fast convolution with embedded constraints to apply FFT acceleration to PD transient diffusion in nonperiodic domains, demonstrating 3D  simulations with more than one billion nodes. This advance substantially reduces the arithmetic work of neighborhood summation, but its computational advantage has not been fully translated into parallel scaling gains. Their method also uses GPU and multi-CPU acceleration; however, for the reported problem with approximately 134 million nodes, increasing the CPU count from eight to sixteen reduces the runtime only from 4.5 to 3.9 hours, a speedup of approximately \(1.15\)~\cite{Jafarzadeh2021FCBM}.

Combining the low arithmetic complexity of FFTs with the parallel computing capacity of multiple GPUs is a natural choice for further improving solver efficiency. Previous studies have pursued this approach by developing distributed multi-GPU FFT methods, but simply distributing the computational workload across more GPUs has not automatically yielded proportional speedups~\cite{Gholami2016AccFFT,Ayala2020heFFTe,Verma2023GPUFFT}. The key obstacle is that full 3D  FFTs require global data redistribution when performing transforms along different directions: as the GPU count increases, the local computational workload per GPU decreases, whereas inter-GPU data exchange and synchronization costs are difficult to reduce proportionally, severely limiting strong scalability~\cite{Ayala2019MPIFFT,Ayala2021FFTScalability,Ayala2022MultiGPUFFT,Czechowski2012FFTCommunication,IbeidOlsonGropp2020Exascale}. In their strong-scaling analysis, Ayala et al. further point out that network saturation reduces the average bandwidth available to each process, ultimately causing linear strong scaling to break down~\cite{Ayala2022MultiGPUFFT}. To mitigate this problem, researchers have adopted different domain decomposition schemes, communication optimizations, and strategies for overlapping computation and communication~\cite{Ayala2019MPIFFT,Ayala2021FFTScalability,Ayala2022MultiGPUFFT,Pekurovsky2012P3DFFT,Gholami2016AccFFT,Ayala2020heFFTe,Verma2023GPUFFT}. However, these optimizations retain the global data redistribution required by the full 3D  transform, leaving the communication bottleneck unresolved. For example, for a fixed \(1024^3\) transform, AccFFT, which uses pencil decomposition and communication overlap, achieves only 26\% strong-scaling efficiency when the GPU count increases from 128 to 4,096~\cite{Gholami2016AccFFT}; after adopting slab decomposition and asynchronous communication optimizations, Verma et al. report that, for a double-precision transform of the same size, the fraction of total runtime spent on communication rises sharply from 51\% to 98\% as the number of A100 GPUs increases from 8 to 128~\cite{Verma2023GPUFFT}. Ayala et al. also explicitly identify the current need to develop algorithms that reduce communication to achieve further speedups while maintaining good scalability~\cite{Ayala2019MPIFFT}. Thus, simply combining FFTs with GPU computing does not yield the efficient solver sought here. The key challenge is how to preserve the original discrete operator and the computational advantages of FFTs while avoiding the global data redistribution introduced by full 3D  FFTs, and to translate the finite interaction range of the nonlocal operator into a local communication structure suited to multi-GPU parallelism.

To address this problem, we exploit the finite interaction range of the ACFD operator to develop horizon-local direct (HLD) evaluation. We observe that HLD exchanges data only between neighboring subdomains and maintains good multi-GPU scaling, although direct neighborhood summation still limits its computational efficiency. This observation motivates our sliced FFT algorithm, which combines the local communication of HLD with fast FFT-based computation for multi-GPU execution. By reorganizing the convolution evaluation of the ACFD operator, the algorithm reduces the computational cost of direct neighborhood summation while avoiding the global data redistribution required by full 3D  FFTs, thereby combining fast operator evaluation with efficient multi-GPU scaling. The original ACFD discrete operator is preserved throughout, retaining the accuracy and asymptotic compatibility of the adopted scheme.

We build matrix-free GPU solvers around these operator algorithms to avoid explicit assembly and storage of the global stiffness matrix, and use multigrid-preconditioned conjugate gradients (MG--PCG) to reduce iteration counts in implicit solves~\cite{ChenCaoSerra2024FastAMG,Naumov2015AmgX}. Multigrid introduces a configuration problem across levels: coarsening changes the problem size, relative interaction range, and balance between computation and communication, so suitable algorithms, GPU counts, and arithmetic precisions~\cite{McCormick2021MixedPrecision} differ between levels. With \(k\) choices per level and \(L\) levels, there are \(k^L\) combinations before cross-level constraints are imposed. These choices must also accommodate memory limits and interlevel data transfers, making manual configuration testing costly in implementation and experimentation.

To address this combinatorial problem, we develop the human-audited \mbox{MG-Pilot} agent for joint configuration optimization across levels, targeting the time and GPU memory use of the complete solve. Scientific agents for task orchestration and execution feedback~\cite{Ding2025SciToolAgent,Shao2026SciSciGPT}, together with CUDA optimization driven by hardware feedback~\cite{Zhang2025CudaForge}, provide the tools for this approach. \mbox{MG-Pilot} screens out combinations that violate grid and resource constraints, invokes the GPU solver automatically, and updates candidate configurations using measured solution time, peak memory use, and numerical convergence. Human auditing checks numerical results and timing scope, after which the retained configurations are validated at the target scale, yielding time-oriented choices for faster solves and capacity-oriented choices for larger simulations.

The main contributions are as follows.\nopagebreak[4]
\begin{itemize}
\item We propose a fast and scalable multi-GPU sliced FFT algorithm. Exploiting the communication locality of HLD, the algorithm combines local communication with fast FFT-based computation and an ACFD scheme, and avoids the global data redistribution required by full 3D  FFTs to remain  fast and scalable for multi-GPU parallel computation.

\item We construct a matrix-free GPU solver framework for the ACFD scheme, avoiding explicit storage of the global stiffness matrix. Multigrid-preconditioned conjugate gradients (MG--PCG) reduce iteration counts and enable efficient solution of large discrete systems.

\item We develop the \mbox{MG-Pilot} agent to address the combinatorial configuration problem across levels in multi-GPU multigrid. It combines feasibility screening, GPU performance feedback, and numerical validation under human auditing to jointly select operator algorithms, GPU counts, arithmetic precisions, and coarsening strategies. The resulting optimized configurations target either shorter solution times or larger attainable problem sizes.
\end{itemize}

Numerical experiments demonstrate nearly linear strong scaling of HLD--CG and sliced FFT--CG on up to eight GPUs. On the same problem and discretization, single-GPU sliced FFT--MG--PCG achieves a speedup of approximately \(1{,}020\times\) over 28-thread CPU--CG. The framework supports approximately 10.29 billion degrees of freedom on eight RTX~4090 GPUs with 24~GB of memory each. Compared with previous work, we increase the attainable scale and reduce computational time for the insulated-cutout benchmark~\cite{Jafarzadeh2021FCBM}. 3D  star-shaped interface evolution further demonstrates efficient large-scale computation for a nonlinear phase-field problem and approximately second-order convergence toward the local solution.

The remainder of this paper is organized as follows. Section~2 introduces the nonlocal diffusion model and its ACFD discretization. Section~3 develops the full FFT, HLD, and sliced FFT algorithms and their multi-GPU formulations. Section~4 presents the multigrid solver and agent-assisted configuration optimization. Section~5 reports numerical experiments for transient diffusion and phase-field problems. Section~6 concludes the paper.

\section{An asymptotically compatible discretization scheme}
\label{sec:acfd-discretization}

We use the nonlocal diffusion operator
\begin{equation}\label{eq:nonlocal-operator}
\mathcal L_\delta u(\bm x)
=2\int_{B_\delta(\bm 0)}
\bigl[u(\bm x+\bm z)-u(\bm x)\bigr]
\rho_\delta(|\bm z|)\,\mathrm d\bm z,
\end{equation}
where \(B_\delta(\bm 0)\) is the ball of interaction radius \(\delta\), \(\bm z\) is the relative displacement, and \(\rho_\delta(|\bm z|)\) is a nonnegative radial kernel supported in \(B_\delta(\bm 0)\). We consider the nonlocal diffusion problem
\begin{equation}\label{eq:continuous-problem}
\begin{cases}
-\mathcal L_\delta u_\delta(\bm x)=f(\bm x),&\bm x\in\Omega,\\
u_\delta(\bm x)=g(\bm x),&\bm x\in\Omega_c.
\end{cases}
\end{equation}
Here, the nonlocal operator is defined in~\eqref{eq:nonlocal-operator}, \(\Omega\subset\mathbb R^3\) is a bounded domain, and \(\Omega_c=\{\bm x\notin\Omega:\operatorname{dist}(\bm x,\partial\Omega)<\delta\}\) is the exterior layer on which the volume constraint is prescribed. The unknown \(u_\delta\) is a scalar field, while \(f\) and \(g\) are the source term and prescribed volume-constraint data, respectively. The kernel satisfies the second-moment condition \(\int_{B_\delta(\bm0)}|\bm z|^2\rho_\delta(|\bm z|)\,\mathrm d\bm z=3\). Under this condition and as \(\delta\to0\), problem~\eqref{eq:continuous-problem} converges to the corresponding local Poisson problem~\cite{Du2012NonlocalDiffusion,DuBook2019}:
\[
-\Delta u=f.
\]

This convergence from the nonlocal model to the classical local model as \(\delta\to0\) is known as the local limit. A numerical scheme is also expected recover this limit, i.e.,  as both \(\delta\) and  \(h \rightarrow 0\) together, its solution should converge to that of the corresponding local problem. Such a scheme is called asymptotically compatible~\cite{TianDu2014AC,TianDu2020Review}. 
It is a central requirement for reliable nonlocal discretizations and provides a theoretical basis for model validation, parameter calibration, and multiscale simulation. We adopt the recently developed ACFD scheme in~\cite{Ding2026Approxcaps}. Based on a uniform grid and an approximate-ball construction, this scheme achieves second-order convergence and preserves asymptotic compatibility. It also permits efficient stiffness matrix assembly and yields a symmetric positive definite discrete system. The discretization is given below.

\begin{figure}[!htbp]
\centering
\includegraphics[width=0.8\linewidth]{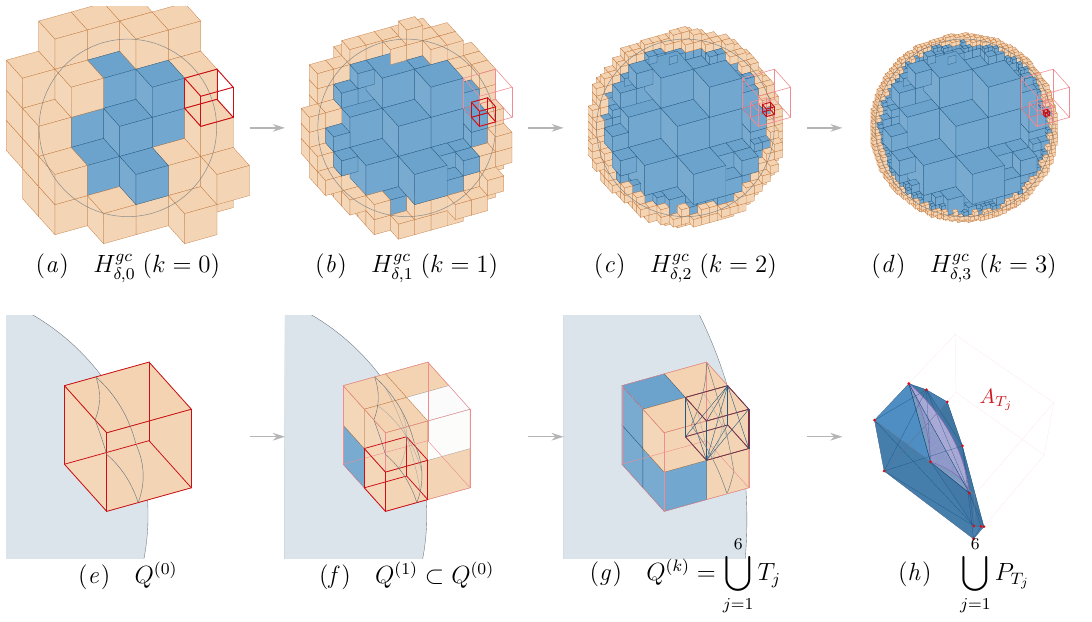}
\caption{Construction of the grid-cell approximate ball: classification of background cells, recursive subdivision of cells intersected by the sphere, and clipping at the final refinement level.}
\label{fig:grid-cell-approximate-ball}
\end{figure}

On a uniform Cartesian grid with spacing \(h\), let \(\bm i=(i_1,i_2,i_3)\in\mathbb Z^3\) denote a grid multi-index, with nodes \(\bm x_{\bm i}=h\bm i\) and nodal values \(u_{\bm i}=u_\delta(\bm x_{\bm i})\). For \(\bm z\ne\bm0\), define \(\omega(\bm z)=|\bm z|^2/\|\bm z\|_1\) and the weighted difference quotient \(q_{\bm i}(\bm z)=[u(\bm x_{\bm i}+\bm z)-u(\bm x_{\bm i})]/\omega(\bm z)\). For each nonzero integer offset \(\bm\xi\), let \(\phi_{\bm\xi}\) be the piecewise trilinear basis function associated with the node \(h\bm\xi\), and let \(I_h\) denote the corresponding interpolation operator. Then
\begin{equation}\label{eq:acfd-interpolation}
I_hq_{\bm i}(\bm z)
=\sum_{\bm\xi\ne\bm0}
\frac{u_{\bm i+\bm\xi}-u_{\bm i}}{\omega(h\bm\xi)}
\phi_{\bm\xi}(\bm z).
\end{equation}

In 3D discretizations, the spherical integration domain is commonly approximated by a polyhedron to facilitate integration over cells intersected by the sphere. We use the grid-cell approxcaps strategy of~\cite{Ding2026Approxcaps}, with \(k\) levels of geometric refinement applied to these cells. The resulting polyhedral approximation inscribed in the interaction ball \(B_\delta(\bm0)\) is denoted by \(H_{\delta,k}^{\gc}\). The approximation satisfies \(H_{\delta,k}^{\gc}\subset B_\delta(\bm0)\), and the refinement level \(k\) controls its geometric error. In the numerical experiments, we use the same refinement level \(k\) as in~\cite{Ding2026Approxcaps}. Replacing the exact ball by this region and substituting~\eqref{eq:acfd-interpolation} into the integral gives
\begin{equation}\label{eq:acfd-stencil}
\begin{aligned}
(\mathcal L_{\delta,h,k}^{\gc}u)_{\bm i}
&=2\int_{H_{\delta,k}^{\gc}}
I_hq_{\bm i}(\bm z)
\omega(\bm z)\rho_\delta(|\bm z|)\,\mathrm d\bm z\\
&=\sum_{\bm\xi\ne\bm0}
\left[
\frac{2}{\omega(h\bm\xi)}
\int_{H_{\delta,k}^{\gc}}
\phi_{\bm\xi}(\bm z)\omega(\bm z)
\rho_\delta(|\bm z|)\,\mathrm d\bm z
\right]
\bigl(u_{\bm i+\bm\xi}-u_{\bm i}\bigr)
:=\sum_{\bm\xi\ne\bm0}a_{\bm\xi}
\bigl(u_{\bm i+\bm\xi}-u_{\bm i}\bigr).
\end{aligned}
\end{equation}
{The coefficients \(a_{\bm\xi}\) depend on \(h\), \(\delta\), and \(k\), but we suppress these dependencies for notational simplicity.}
Denote the coefficients in square brackets by \(a_{\bm\xi}\), and let \(\mathcal S=\{\bm\xi\in\mathbb Z^3\setminus\{\bm0\}:a_{\bm\xi}>0\}\) be the set of integer offsets with nonzero weights. Define the index sets of interior and volume-constraint nodes as \(\mathcal G_h=\{\bm i:\bm x_{\bm i}\in\Omega\}\) and \(\mathcal G_{c,h}=\{\bm i:\bm x_{\bm i}\in\Omega_c\}\), respectively, and write \(f_{\bm i}=f(\bm x_{\bm i})\) and \(g_{\bm i}=g(\bm x_{\bm i})\). {Since the continuous problem~\eqref{eq:continuous-problem} involves \(-\mathcal L_\delta u_\delta=f\), the discrete scheme reverses the sign of the difference in~\eqref{eq:acfd-stencil}:} Combining~\eqref{eq:continuous-problem} and~\eqref{eq:acfd-stencil}, we obtain the complete ACFD scheme: find the nodal values \(U_{\bm i}\) such that
\begin{equation}\label{eq:acfd-full-scheme}
\begin{cases}
\displaystyle
\sum_{\bm\xi\in\mathcal S}
a_{\bm\xi}\bigl(U_{\bm i}-U_{\bm i+\bm\xi}\bigr)=f_{\bm i},
&\bm i\in\mathcal G_h,\\[1.2ex]
U_{\bm i}=g_{\bm i},
&\bm i\in\mathcal G_{c,h}.
\end{cases}
\end{equation}

The discrete scheme~\eqref{eq:acfd-full-scheme} leads to the algebraic system
\begin{equation}\label{AS}
A_h\bm U_h=\bm F_h,
\end{equation}
where the known values at the volume-constraint nodes are incorporated into \(\bm F_h\), and \(A_h\) is a symmetric nonsingular M-matrix~\cite{Du2019ACFD,Ding2026Approxcaps}. {This property guarantees the existence and uniqueness of the discrete solution.
}

\section{Operator reformulation and multi-GPU data dependencies for 3D discrete convolution}

This section develops three efficient GPU algorithms for the algebraic system~\eqref{AS}: full 3D FFT, horizon-local direct (HLD), and sliced FFT.

\subsection{GPU-based full 3D FFT algorithm}\label{sec:full-fft}
To accelerate matrix--vector multiplication in the discrete system of \(A_h\bm U_h\) in~\eqref{AS}, we first examine the structure of matrix  \(A_h\). Let \(d_h=\sum_{\bm\xi\in\mathcal S}a_{\bm\xi}\), and write the central and neighboring coefficients as \(c_{\bm0}=d_h\) and \(c_{\bm\xi}=-a_{\bm\xi}\) for \(\bm\xi\in\mathcal S\), with zero coefficients at all other offsets. Setting \(\bm j=\bm i+\bm\xi\) in~\eqref{eq:acfd-full-scheme} gives
\begin{equation}\label{eq:acfd-toeplitz-entries}
(A_h\bm U_h)_{\bm i}
=d_hU_{\bm i}
-\sum_{\substack{\bm\xi\in\mathcal S\\
\bm i+\bm\xi\in\mathcal G_h}}
a_{\bm\xi}U_{\bm i+\bm\xi}
=\sum_{\bm j\in\mathcal G_h}c_{\bm j-\bm i}U_{\bm j},
\qquad \bm i\in\mathcal G_h.
\end{equation}

On the \(n_x\times n_y\times n_z\) grid of interior nodes, order the unknowns with \(z\) varying fastest, followed by \(y\) and then \(x\). Write \(s=j_x-i_x\), \(t=j_y-i_y\), and \(r=j_z-i_z\). For fixed \(s,t\), the couplings along \(z\) form an \(n_z\times n_z\) matrix \(T_{s,t}\). For fixed \(s\), these matrices form a plane block \(B_s\) along \(y\). Arranging the plane blocks along \(x\) gives the full matrix \(A_h\):
\begingroup
\setlength{\arraycolsep}{2pt}
\begin{equation}\label{eq:three-level-toeplitz}
T_{s,t}=
\begin{bmatrix}
c_{s,t,0}&c_{s,t,1}&c_{s,t,2}&\cdots\\
c_{s,t,-1}&c_{s,t,0}&c_{s,t,1}&\cdots\\
c_{s,t,-2}&c_{s,t,-1}&c_{s,t,0}&\cdots\\
\vdots&\vdots&\vdots&\ddots
\end{bmatrix},
\;
B_s=
\begin{bmatrix}
T_{s,0}&T_{s,1}&T_{s,2}&\cdots\\
T_{s,-1}&T_{s,0}&T_{s,1}&\cdots\\
T_{s,-2}&T_{s,-1}&T_{s,0}&\cdots\\
\vdots&\vdots&\vdots&\ddots
\end{bmatrix},
\;
A_h=
\begin{bmatrix}
B_0&B_1&B_2&\cdots\\
B_{-1}&B_0&B_1&\cdots\\
B_{-2}&B_{-1}&B_0&\cdots\\
\vdots&\vdots&\vdots&\ddots
\end{bmatrix}.
\end{equation}
\endgroup
Each diagonal of \(T_{s,t}\) has a constant coefficient, and each block diagonal of \(B_s\) and \(A_h\) repeats the same block. Thus, \(A_h\) has a multilevel Toeplitz structure, which reduces matrix storage and permits FFT acceleration of matrix--vector products.

Following the size requirements for standard circulant embedding~\cite{VollmannSchulz2019}, zero-pad the interior vector \(\bm U_h\) to a \(Q_x\times Q_y\times Q_z\) array \(U^Q\). The central symmetry of the weights gives \(c_{\bm\xi}=c_{-\bm\xi}\), so~\eqref{eq:acfd-toeplitz-entries} can be written as a convolution:
\begin{equation}\label{eq:stencil-to-linear-convolution}
(A_h\bm U_h)_{\bm i}
=\sum_{\bm j\in\mathcal G_h}c_{\bm j-\bm i}U_{\bm j}
=\sum_{\bm j\in\mathcal G_h}c_{\bm i-\bm j}U_{\bm j}
=(c*_{\rm lin}U^Q)_{\bm i},
\qquad \bm i\in\mathcal G_h.
\end{equation}
Here, \(*_{\rm lin}\) denotes discrete linear convolution. Embed the coefficient array \(c\) in an array \(c^Q\) of the same size, and let \(C\) be the three-level circulant matrix with convolution kernel \(c^Q\). At each level, successive block rows of \(C\) are obtained by cyclic shifts. The product \(Y=CU^Q\) satisfies \(Y_{\bm i}=(A_h\bm U_h)_{\bm i}\) at interior nodes.

Let \(\mathcal F_Q\) and \(\mathcal F_Q^{-1}\) denote the discrete Fourier transform matrix and its inverse. Write \(\bm\kappa=(\kappa_x,\kappa_y,\kappa_z)\) for the frequency index and \(\mathrm i\) for the imaginary unit. Transforming the coefficient array \(c^Q\) gives the kernel spectrum
\begin{equation}\label{eq:full-fft-symbol}
\lambda_{\bm\kappa}
=(\mathcal F_Qc^Q)_{\bm\kappa}
=d_h-\sum_{\bm\xi\in\mathcal S}a_{\bm\xi}
\exp\!\Big[-2\pi\mathrm i
\big(
\frac{\kappa_x\xi_x}{Q_x}
+\frac{\kappa_y\xi_y}{Q_y}
+\frac{\kappa_z\xi_z}{Q_z}
\big)\Big].
\end{equation}
\begin{samepage}
Arrange \(\lambda_{\bm\kappa}\) along the main diagonal in frequency order to form the diagonal matrix \(\Lambda\). Fourier diagonalization of the three-level circulant matrix gives~\cite{VollmannSchulz2019,Gray2006Toeplitz}
\begin{equation}\label{eq:full-fft-diagonalization}
\mathcal F_QC\mathcal F_Q^{-1}
=\Lambda
=\begin{bmatrix}
\lambda_{0,0,0}&0&0&\cdots\\
0&\lambda_{0,0,1}&0&\cdots\\
0&0&\lambda_{0,0,2}&\cdots\\
\vdots&\vdots&\vdots&\ddots
\end{bmatrix}.
\end{equation}
\end{samepage}

Denote the input and output spectra by \(\widehat U^Q=\mathcal F_QU^Q\) and \(\widehat Y=\mathcal F_QY\), respectively. Equation~\eqref{eq:full-fft-diagonalization} yields \(\widehat Y=\mathcal F_QCU^Q=\mathcal F_QC\mathcal F_Q^{-1}\widehat U^Q=\Lambda\widehat U^Q\). Each component therefore requires only the product \(\widehat Y_{\bm\kappa}=\lambda_{\bm\kappa}\widehat U^Q_{\bm\kappa}\). Applying the inverse FFT gives \(Y=\mathcal F_Q^{-1}\widehat Y\); extracting its interior-node components recovers \(A_h\bm U_h\). Thus, each matrix--vector product required to solve~\eqref{AS} uses one forward FFT, componentwise multiplication, and one inverse FFT, substantially reducing the cost of direct neighborhood summation.

On multiple GPUs, multiplication by \(\Lambda\) is readily parallelized, but the forward and inverse Fourier transforms still require inter-GPU communication. When \(U^Q\) is distributed along \(x\), each GPU holds a subset of complete \(yz\) planes. Let \(U^Q_{j_x}\) denote plane \(j_x\) of the extended array, and let \(\mathcal F_{yz}\) be the two-dimensional discrete Fourier transform matrix on a \(Q_y\times Q_z\) plane. Using zero-based indices for the extended array, the full 3D  transform satisfies
\begin{equation}\label{eq:full-fft-global-dependence}
(\mathcal F_QU^Q)(\kappa_x,\kappa_y,\kappa_z)
=\sum_{j_x=0}^{Q_x-1}
\exp\!\left(-\frac{2\pi\mathrm i\kappa_xj_x}{Q_x}\right)
(\mathcal F_{yz}U^Q_{j_x})(\kappa_y,\kappa_z),
\end{equation}
where \(0\le\kappa_\mu<Q_\mu\) for \(\mu=x,y,z\). Thus, even after the plane-wise transforms have been completed locally, each output frequency along \(x\) still depends on input planes distributed across all subdomains. The finite interaction range of the original discrete operator therefore does not make the full 3D  transform communication-local. Conventional distributed FFTs perform these transforms through global redistribution of volume data. Increasing the GPU count reduces the local arithmetic work but retains cross-subdomain data exchange, creating a major bottleneck for strong scaling. To reduce this communication, we next exploit the finite interaction range of the discrete operator to construct HLD.

\subsection{HLD algorithm}
HLD evaluates matrix--vector products directly within the finite interaction range of the discrete operator. Distribute the grid along \(x\) over \(P\) GPUs, indexed by \(p=0,\ldots,P-1\). Let \(\ell_p\) and \(u_p\) be the global indices of the first and last \(x\) planes owned by GPU \(p\). For interior nodes with \(\ell_p\le i_x\le u_p\), split~\eqref{eq:acfd-toeplitz-entries} according to whether the input node is local as
\begin{equation}\label{eq:hld-local-remote-split}
(A_h\bm U_h)_{\bm i}
=
\sum_{\substack{\bm j\in\mathcal G_h\\
\ell_p\le j_x\le u_p}}
c_{\bm j-\bm i}U_{\bm j}
+
\sum_{\substack{\bm j\in\mathcal G_h\\
j_x<\ell_p\ \mathrm{or}\ j_x>u_p}}
c_{\bm j-\bm i}U_{\bm j}.
\end{equation}
\begin{figure}[!htb]
\centering
\includegraphics[width=\textwidth]{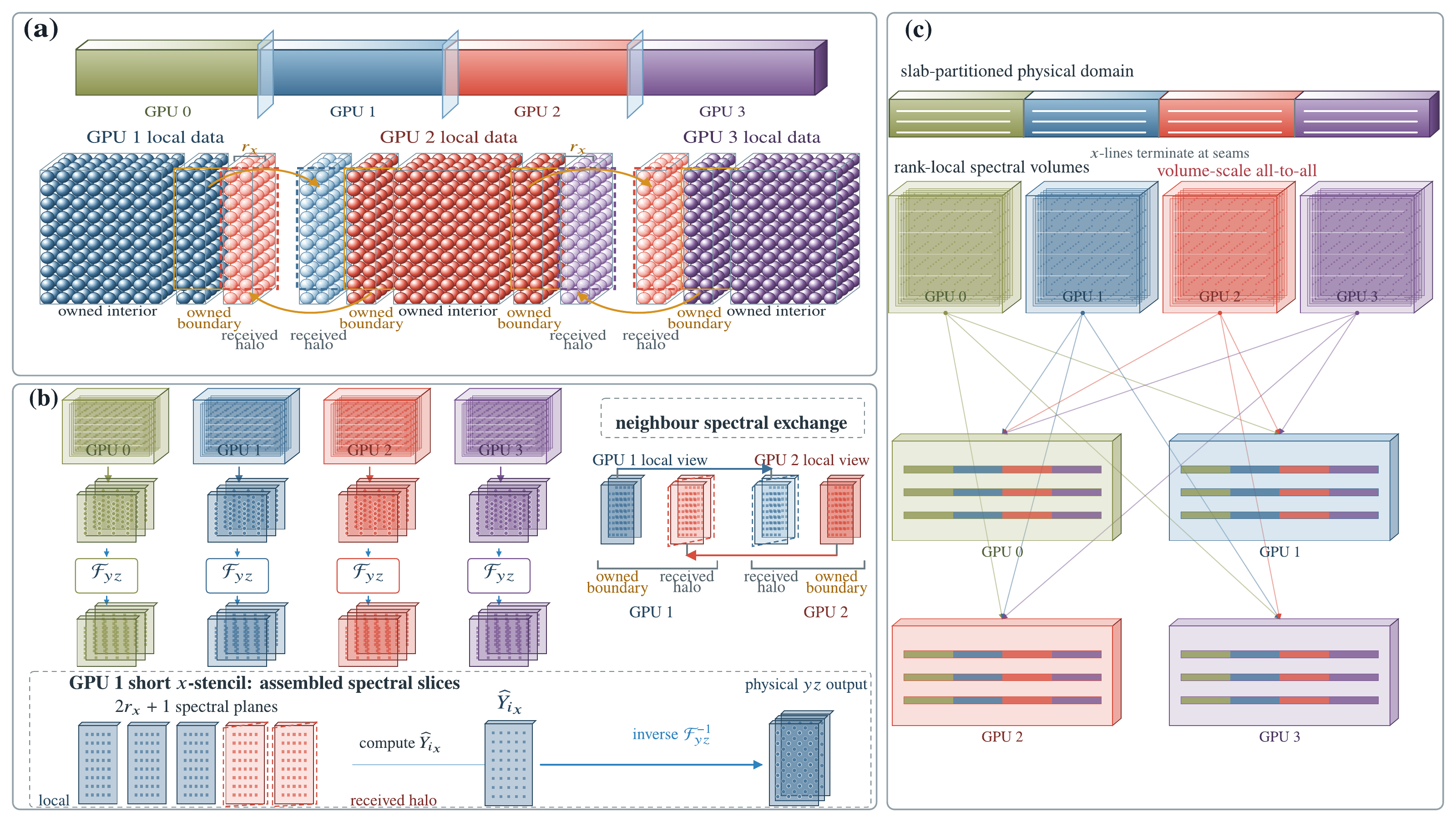}
\caption{Multi-GPU data dependencies under slab decomposition. (a) HLD exchanges physical-space halos between neighboring subdomains. (b) Sliced FFT exchanges spectral halos after plane-wise Fourier transforms. (c) Full 3D FFT requires global redistribution of volume data.}
\label{fig:multigpu-operators}
\end{figure}
The first term uses local data; the second requires nodal values from other GPUs. Define the half-bandwidth along \(x\) by
\[
r_x=\max_{\bm\xi\in\mathcal S}|\xi_x|.
\]

Since \(c_{\bm j-\bm i}=0\) when \(|j_x-i_x|>r_x\), remote inputs for the local planes \(\ell_p\le i_x\le u_p\) lie only in \([\ell_p-r_x,\ell_p-1]\) and \([u_p+1,u_p+r_x]\). If each subdomain contains at least \(r_x\) planes, these two ranges belong to the left and right neighboring subdomains. Each GPU therefore needs only to receive these values from its neighbors and store them in local halo buffers to complete its matrix--vector product.

In the implementation, each GPU initiates halo exchange on both sides. If the local plane count exceeds \(2r_x\), nodes independent of remote data can be processed while communication proceeds. Once the halo data arrive, only the two boundary layers remain to be computed. This overlaps local communication with computation. Each GPU accumulates the contributions in~\eqref{eq:hld-local-remote-split}; the resulting local segments together form \(A_h\bm U_h\).

HLD exchanges data only between neighboring subdomains, avoiding the global redistribution required by full 3D FFTs. It still accumulates neighborhood contributions directly. To further reduce computational cost, we next combine FFT acceleration with this local communication structure.

\subsection{Sliced FFT algorithm}
The sliced FFT algorithm combines the computational efficiency of FFTs with the local communication of HLD. It applies Fourier transforms to the plane blocks while retaining their coupling along \(x\). Let \(U_{i_x}\) denote the vector of nodal values of \(\bm U_h\) on plane \(i_x\). The block structure in~\eqref{eq:three-level-toeplitz} gives
\begin{equation}\label{eq:sliced-fft-spatial}
(A_h\bm U_h)_{i_x}
=\sum_{s=-r_x}^{r_x}B_sU_{i_x+s},
\end{equation}
where each plane block \(B_s\) has a two-level Toeplitz structure, and \(B_s=0\) for \(|s|>r_x\). Using the standard circulant embedding described in Section~\ref{sec:full-fft}, zero-pad each \(U_{i_x}\) along \(y,z\) to a \(Q_y\times Q_z\) array \(U^Q_{i_x}\), and embed \(B_s\) in a two-level circulant matrix \(C_s\). Retain the original plane indices along \(x\), and denote the output on each extended plane by \(Y_{i_x}=\sum_{s=-r_x}^{r_x}C_sU^Q_{i_x+s}\).

Use the plane Fourier transform matrix \(\mathcal F_{yz}\) defined in \eqref{eq:full-fft-global-dependence}. For a fixed plane offset \(s\), the coefficients \(c_{s,t,r}\) of \(B_s\) give
\begin{equation}\label{eq:sliced-block-symbol}
\lambda_s(\kappa_y,\kappa_z)
=\sum_{t,r\in\mathbb Z}c_{s,t,r}
\exp\!\Big[
2\pi\mathrm i
\big(
\frac{\kappa_y t}{Q_y}
+\frac{\kappa_z r}{Q_z}
\big)
\Big],
\qquad
0\le\kappa_y<Q_y,\quad 0\le\kappa_z<Q_z.
\end{equation}
Placing these spectral coefficients on the main diagonal gives a diagonal matrix \(\Lambda_s\), with
\[
\mathcal F_{yz}C_s\mathcal F_{yz}^{-1}=\Lambda_s.
\]

Apply this transform separately to all planes. The overall transform matrix is \(\mathcal F_{\rm slice}=\operatorname{diag}(\mathcal F_{yz},\ldots,\mathcal F_{yz})\), with \(n_x\) identical diagonal blocks. Hence,
\begingroup
\setlength{\arraycolsep}{4pt}
\begin{equation}\label{eq:sliced-global-factorization}
\mathcal F_{\rm slice}
\begin{bmatrix}
C_0&C_1&C_2&\cdots\\
C_{-1}&C_0&C_1&\cdots\\
C_{-2}&C_{-1}&C_0&\cdots\\
\vdots&\vdots&\vdots&\ddots
\end{bmatrix}
\mathcal F_{\rm slice}^{-1}
=
\begin{bmatrix}
\Lambda_0&\Lambda_1&\Lambda_2&\cdots\\
\Lambda_{-1}&\Lambda_0&\Lambda_1&\cdots\\
\Lambda_{-2}&\Lambda_{-1}&\Lambda_0&\cdots\\
\vdots&\vdots&\vdots&\ddots
\end{bmatrix}.
\end{equation}
\endgroup
Each transformed plane block is diagonal. Since \(C_s=0\) for \(|s|>r_x\), we also have \(\Lambda_s=0\) for these offsets. The overall matrix therefore retains its block Toeplitz structure with half-bandwidth \(r_x\) along \(x\).
Denote the input and output plane spectra by \(\widehat U^Q_{i_x}=\mathcal F_{yz}U^Q_{i_x}\) and \(\widehat Y_{i_x}=\mathcal F_{yz}Y_{i_x}\), respectively. Equation~\eqref{eq:sliced-global-factorization} gives
\begin{equation}\label{eq:sliced-spectral-action}
\widehat Y_{i_x}
=\mathcal F_{yz}\sum_{s=-r_x}^{r_x}C_sU^Q_{i_x+s}
=\sum_{s=-r_x}^{r_x}
\bigl(\mathcal F_{yz}C_s\mathcal F_{yz}^{-1}\bigr)
\widehat U^Q_{i_x+s}
=\sum_{s=-r_x}^{r_x}\Lambda_s\widehat U^Q_{i_x+s}.
\end{equation}
Because \(\Lambda_s\) is diagonal, this reduces to a separate sum at each frequency:
\begin{equation}\label{eq:sliced-fft-operator}
\widehat Y_{i_x}(\kappa_y,\kappa_z)
=\sum_{s=-r_x}^{r_x}
\lambda_s(\kappa_y,\kappa_z)
\widehat U^Q_{i_x+s}(\kappa_y,\kappa_z).
\end{equation}

Each output spectral component requires data from at most \(2r_x+1\) neighboring planes. An inverse Fourier transform on each plane gives \(Y_{i_x}=\mathcal F_{yz}^{-1}\widehat Y_{i_x}\). Extracting the interior-node components recovers \(A_h\bm U_h\). For fixed weights and a fixed grid level, \(\lambda_s\) is precomputed once and reused. Each matrix--vector product then requires plane-wise forward FFTs, the multiplications and summations in~\eqref{eq:sliced-fft-operator}, and plane-wise inverse FFTs.

After decomposition along \(x\), each GPU owns complete \(yz\) planes. The block-diagonal transform \(\mathcal F_{\rm slice}\) acts independently on these planes, so both forward and inverse FFTs are local, and the spectra remain on the same GPU. The transformed matrix retains half-bandwidth \(r_x\) along \(x\). By~\eqref{eq:sliced-fft-operator}, computing outputs on \(\ell_p\le i_x\le u_p\) requires input spectra only from \([\ell_p-r_x,u_p+r_x]\). The only nonlocal inputs are the \(r_x\) planes immediately beyond each local boundary. If each subdomain contains at least \(r_x\) planes, the left and right neighboring GPUs supply these spectra to the local halo buffers. Inputs and outputs retain the original plane distribution, so no global redistribution is needed. When the local plane count exceeds \(2r_x\), interior-plane computation overlaps spectral exchange, and the two boundary layers are processed after the data arrive. Plane-block diagonalization thus accelerates neighborhood summation while preserving local communication along \(x\). Figure~\ref{fig:multigpu-operators} compares the data dependencies of the three algorithms. Algorithm~\ref{alg:sliced-fft} summarizes the parallel sliced FFT procedure.

\begingroup
\setlength{\intextsep}{4pt}
\begin{algorithm}[H]
\caption{Multi-GPU matrix--vector multiplication by sliced FFT}
\label{alg:sliced-fft}
\small
\begin{algorithmic}[1]
\REQUIRE Coefficients \(c_{s,t,r}\), interior vector \(\bm U_h\), plane extension sizes \(Q_y,Q_z\), and plane ranges \([\ell_p,u_p]\) on each GPU; each subdomain contains at least \(r_x\) planes.
\ENSURE Components of \(A_h\bm U_h\) at the interior nodes owned by each GPU.
\STATE Precompute \(\lambda_s\), \(s=-r_x,\ldots,r_x\), using~\eqref{eq:sliced-block-symbol}; reuse them for fixed weights and a fixed grid level.
\STATE On all GPUs in parallel, set spectra outside the global interior-plane range to zero and proceed as follows.
\STATE Zero-pad each local plane \(U_{i_x}\) along \(y,z\), and compute \(\widehat U^Q_{i_x}=\mathcal F_{yz}U^Q_{i_x}\).
\STATE Initiate exchange with neighboring GPUs: send spectra from the \(r_x\) planes at each end and receive spectral halos.
\STATE During communication, apply~\eqref{eq:sliced-fft-operator} to the interior planes \(\ell_p+r_x\le i_x\le u_p-r_x\).
\STATE Once the spectral halos arrive, apply~\eqref{eq:sliced-fft-operator} to the remaining local planes.
\STATE Compute \(Y_{i_x}=\mathcal F_{yz}^{-1}\widehat Y_{i_x}\) and extract the interior-node components on each local plane.
\end{algorithmic}
\end{algorithm}
\endgroup

The following propositions summarize the algebraic exactness, bandwidth preservation, and computational complexity of this construction. We compare the matrix product and the FFT result directly at the original interior nodes. We set \(U_{i_x}\) and its zero-padded extension to zero outside the global interior-plane range.

\begin{proposition}[Algebraic exactness and bandwidth preservation]\label{prop:sliced-exactness}
Let \(C_s\) be an alias-free circulant embedding of \(B_s\), with \(\Lambda_s=\mathcal F_{yz}C_s\mathcal F_{yz}^{-1}\). In exact arithmetic, sliced FFT satisfies
\begin{equation}\label{eq:sliced-exactness}
(A_h\bm U_h)_{i_x,j,k}
=\left[
\mathcal F_{yz}^{-1}
\left(\sum_{s=-r_x}^{r_x}
\Lambda_s\mathcal F_{yz}U^Q_{i_x+s}\right)
\right]_{j,k}
\end{equation}
for every \(\bm U_h\) at the original interior nodes \(1\le i_x\le n_x\), \(1\le j\le n_y\), \(1\le k\le n_z\). Moreover, \(\Lambda_s=0\) for \(|s|>r_x\), so the plane-wise Fourier transform preserves the block half-bandwidth along \(x\).
\end{proposition}

\begin{proof}
An alias-free circulant embedding makes circular convolution agree with linear convolution at the original interior nodes. For example,
\[
Q_\mu\ge n_\mu+2\max_{\bm\xi\in\mathcal S}|\xi_\mu|,
\qquad \mu=y,z,
\]
is a sufficient choice of extension sizes. For \(1\le j\le n_y\) and \(1\le k\le n_z\), zero-padding and circulant embedding give
\begin{equation}\label{eq:sliced-block-restriction}
\begin{aligned}
(C_sU^Q_{i_x+s})_{j,k}
&=\sum_{j'=1}^{n_y}\sum_{k'=1}^{n_z}
c_{s,j'-j,k'-k}(U_{i_x+s})_{j',k'} =(B_sU_{i_x+s})_{j,k}.
\end{aligned}
\end{equation}
Using \(\mathcal F_{yz}^{-1}\Lambda_s\mathcal F_{yz}=C_s\), equations~\eqref{eq:sliced-block-restriction} and~\eqref{eq:sliced-fft-spatial} yield
\begin{equation}\label{eq:sliced-exactness-proof}
\begin{aligned}
\left[
\mathcal F_{yz}^{-1}
\left(\sum_{s=-r_x}^{r_x}\Lambda_s\mathcal F_{yz}U^Q_{i_x+s}\right)
\right]_{j,k}
&=\sum_{s=-r_x}^{r_x}(C_sU^Q_{i_x+s})_{j,k} =\sum_{s=-r_x}^{r_x}(B_sU_{i_x+s})_{j,k}
=(A_h\bm U_h)_{i_x,j,k}.
\end{aligned}
\end{equation}
For \(|s|>r_x\), the finite interaction range implies \(c_{s,t,r}=0\) for all \(t,r\). The circulant block constructed from these coefficients is therefore \(C_s=0\), and hence \(\Lambda_s=0\). Thus, the plane-wise transform does not enlarge the coupling range along the decomposition direction.
\end{proof}

Proposition~\ref{prop:sliced-exactness} shows that sliced FFT changes only how the matrix is applied. It preserves the original ACFD discrete system, including the accuracy and asymptotic compatibility of the scheme and the symmetry and positive definiteness of the matrix.

\begin{proposition}[Computational complexity]\label{prop:sliced-complexity}
Suppose that the transverse padding ratios are bounded, with \(Q_y\le c_y n_y\) and \(Q_z\le c_z n_z\) for fixed constants \(c_y,c_z\). After precomputation of the kernel spectra, the arithmetic complexity of each sliced FFT matrix--vector product is
\begin{equation}\label{eq:sliced-complexity}
\mathcal O\!\left(
N_{\mathrm{dof}}\log N_{\mathrm{dof}}+r_xN_{\mathrm{dof}}
\right),
\qquad N_{\mathrm{dof}}=n_xn_yn_z.
\end{equation}
In particular, for bounded \(r_x\), the total arithmetic work is \(\mathcal O(N_{\mathrm{dof}}\log N_{\mathrm{dof}})\). A balanced decomposition over \(P\le n_x\) GPUs gives \(\mathcal O(N_{\mathrm{dof}}\log N_{\mathrm{dof}}/P)\) arithmetic work per GPU.
\end{proposition}

\begin{proof}
Let \(M=Q_yQ_z\ge2\) be the number of entries in a padded plane, and let \(n_p=u_p-\ell_p+1\) be the number of planes owned by GPU \(p\), so that \(\sum_{p=0}^{P-1}n_p=n_x\). Denote the arithmetic work per matrix--vector product on this GPU by \(W_p\), and the total work by \(W=\sum_pW_p\). Below, \(C\) may change between estimates but is independent of the grid size, \(r_x\), and \(P\).

Each plane requires one forward and one inverse two-dimensional FFT. Each output frequency requires at most \(2r_x+1\) multiply-add operations, while zero-padding and restriction require only linear work. Thus, fixed counting constants \(C_F,C_A,C_Z\) give
\begin{equation}\label{eq:sliced-work-count}
\begin{aligned}
W_p
&\le
\underbrace{2C_Fn_pM(\log Q_y+\log Q_z)}_{\text{forward and inverse plane FFTs}}
\\
&\quad+\underbrace{C_An_pM(2r_x+1)}_{\text{spectral plane coupling}}
+\underbrace{C_Zn_pM}_{\text{padding and restriction}}\\
&\le Cn_pM(1+\log M+r_x)
\le Cn_pM(\log M+r_x),
\end{aligned}
\end{equation}
where we used \(\log Q_y+\log Q_z=\log M\) and \(M\ge2\). Summing over the GPUs yields
\begin{equation}\label{eq:sliced-total-work}
W=\sum_{p=0}^{P-1}W_p
\le CM(\log M+r_x)\sum_{p=0}^{P-1}n_p
=Cn_xM(\log M+r_x).
\end{equation}
For a balanced decomposition over \(P\le n_x\) GPUs, the plane counts differ by at most one, and
\[
n_p\le\left\lceil\frac{n_x}{P}\right\rceil
\le\frac{n_x}{P}+1\le\frac{2n_x}{P}.
\]
Equation~\eqref{eq:sliced-work-count} therefore gives
\begin{equation}\label{eq:sliced-per-gpu-work}
W_p\le C\frac{n_xM}{P}(\log M+r_x).
\end{equation}

Set \(\gamma=c_yc_z\ge1\). The bounded padding ratios imply
\begin{equation}\label{eq:sliced-padding-bound}
\begin{aligned}
M&\le\gamma n_yn_z,
&n_xM&\le\gamma N_{\mathrm{dof}},\\
\log M&\le\log\gamma+\log N_{\mathrm{dof}}
\le C_\gamma\log N_{\mathrm{dof}},
\end{aligned}
\end{equation}
where \(C_\gamma\) depends only on \(\gamma\). Substituting~\eqref{eq:sliced-padding-bound} into~\eqref{eq:sliced-total-work} and absorbing the fixed factors into \(C\) gives
\begin{equation}\label{eq:sliced-work-simplification}
\begin{aligned}
W
&\le Cn_xM(\log M+r_x)\\
&\le C\gamma N_{\mathrm{dof}}
(C_\gamma\log N_{\mathrm{dof}}+r_x)\\
&\le CN_{\mathrm{dof}}(\log N_{\mathrm{dof}}+r_x).
\end{aligned}
\end{equation}
Similarly, equation~\eqref{eq:sliced-per-gpu-work} yields
\[
W_p\le C\frac{N_{\mathrm{dof}}}{P}
(\log N_{\mathrm{dof}}+r_x).
\]
We next estimate the volume of data exchanged between subdomains. Suppose that each subdomain contains at least \(r_x\) planes. Proposition~\ref{prop:sliced-exactness} shows that GPU \(p\) needs input spectra only from \([\ell_p-r_x,u_p+r_x]\) to compute its local outputs. At most \(r_x\) planes lie beyond each local boundary, and each plane contains \(M\) complex values. Let \(V_p^{\mathrm{send}}\) and \(V_p^{\mathrm{recv}}\) denote the numbers of complex values sent and received by each GPU, respectively. Then
\begin{equation}\label{eq:sliced-communication-volume}
\begin{aligned}
V_p^{\mathrm{send}}&\le2r_xM,\qquad
V_p^{\mathrm{recv}}\le2r_xM,\\
V_p&:=V_p^{\mathrm{send}}+V_p^{\mathrm{recv}}
\le4r_xM=\mathcal O(r_xM).
\end{aligned}
\end{equation}
All these spectra are supplied by neighboring GPUs, while the forward and inverse transforms are performed locally. Inter-GPU communication therefore involves only neighboring spectral halos, avoiding the global redistribution of volume data required by full 3D  FFT in~\eqref{eq:full-fft-global-dependence}.

For bounded \(r_x\), the linear work of plane coupling is absorbed into the FFT term. Combining the estimates above gives the following bounds for the total arithmetic work and the work per GPU under balanced decomposition:
\[
W=\mathcal O(N_{\mathrm{dof}}\log N_{\mathrm{dof}}),
\qquad
\max_{0\le p<P}W_p
=\mathcal O\!\left(\frac{N_{\mathrm{dof}}\log N_{\mathrm{dof}}}{P}\right).
\]
Under the subdomain-thickness condition above, each GPU exchanges only \(\mathcal O(M)\) complex values with its neighbors.
\end{proof}

\section{Agent-assisted joint optimization across multigrid levels}
\label{sec:multigrid-optimization}

Section~3 developed fast matrix--vector multiplication algorithms for multiple GPUs. We now combine these algorithms with multigrid preconditioning to form an MG--PCG solver and jointly optimize its configurations for the computational and memory requirements at different grid levels. To this end, we build a configuration agent, MG-Pilot, on OpenAI Codex using the GPT-5.6 Sol model. Within a prescribed experimental budget, it uses feasibility checks, pilot experiments, and human audits to select time- and capacity-oriented configurations.

\subsection{Multi-GPU MG--PCG solver}

The multi-GPU solver uses preconditioned conjugate gradients as the outer iteration and a multigrid V-cycle as the preconditioner. The fast algorithms in Section~3 evaluate matrix--vector products in both the outer iteration and the multigrid levels. These algorithm choices form the basis of the levelwise configuration optimization below.

For the algebraic system~\eqref{AS}, let \(\bm U_h^{(k)}\) be the approximate solution at iteration \(k\), with residual \(\bm r_k=\bm F_h-A_h\bm U_h^{(k)}\). One symmetric multigrid V-cycle gives the preconditioned residual \(\bm z_k=\mathcal V_0(\bm r_k)\). Starting from \(\bm p_0=\bm z_0\), the PCG iteration reads~\cite{Barrett1994Templates}: \(\bm q_k=A_h\bm p_k\) and 
\begin{equation}\label{eq:pcg-iteration}
\alpha_k=\frac{\bm r_k^T\bm z_k}{\bm p_k^T\bm q_k}, \; 
\bm U_h^{(k+1)}=\bm U_h^{(k)}+\alpha_k\bm p_k, \; 
\bm r_{k+1}=\bm r_k-\alpha_k\bm q_k,  \;
\bm p_{k+1}=\bm z_{k+1}+\beta_k\bm p_k, \;
\beta_k=\frac{\bm r_{k+1}^T\bm z_{k+1}}{\bm r_k^T\bm z_k}.
\end{equation}
Here, \(\alpha_k\) is the step length, and \(\beta_k\) is the coefficient used to update the search direction.

The product \(\bm q_k=A_h\bm p_k\) in~\eqref{eq:pcg-iteration} is computed by the fast algorithms in Section~3. For example, full 3D FFT first zero-pads \(\bm p_k\) to \(p_k^Q\), as in Section~\ref{sec:full-fft}, and then evaluates
\begin{equation}\label{eq:pcg-fft-product}
Y_k=Cp_k^Q
=\mathcal F_Q^{-1}[\Lambda(\mathcal F_Qp_k^Q)],
\qquad
(\bm q_k)_{\bm i}=(Y_k)_{\bm i},
\quad \bm i\in\mathcal G_h.
\end{equation}

HLD or sliced FFT instead computes \(\bm q_k\) by the method in Section~\hyperlink{subsection.3.2}{3.2} or~\hyperlink{subsection.3.3}{3.3}, respectively. The remaining iteration steps are unchanged.

The multigrid preconditioner successively coarsens the finest grid by a factor of two in each direction. At each level, the discrete matrix \(A_\ell\) is constructed from~\eqref{eq:acfd-full-scheme}, where \(\ell=0,\ldots,J\), \(A_0=A_h\), and \(J\) is the coarsest-level index. Jacobi smoothing updates the approximate correction from the residual, using the algorithms in Section~3 for its matrix--vector products. The relaxation factor is 0.8, with one pre-smoothing step and one post-smoothing step. The coarsest level uses 20 Jacobi iterations. Computations across levels follow a standard symmetric V-cycle~\cite{CuiKanschat2025GPU,Briggs2000Multigrid}.

Within this solver framework, the next subsection defines the levelwise configuration variables, constraints, and optimization objectives.

\subsection{Levelwise configuration space and optimization objectives}
\label{sec:configuration-space}

The performance of the MG--PCG solver depends on how computation and memory are managed across the multigrid hierarchy. As the grid is coarsened, the workload decreases and the balance between computation and communication changes, motivating different configurations at different levels~\cite{Reisner2018Redistribution}. We therefore jointly select the operator algorithm, the number of active GPUs, and the internal arithmetic precision at each level, together with the coarsest grid size, to balance solution time and memory use subject to resource and convergence constraints.

We first consider two levelwise choices: the operator algorithm and the number of active GPUs. Let \(B_\ell\in\{H,S\}\) denote the algorithm used at V-cycle level \(\ell\), where \(H\) and \(S\) represent HLD and sliced FFT, respectively. The matrix--vector multiplication algorithm for outer PCG is specified separately according to problem size and available GPU memory. HLD requires less memory, whereas sliced FFT reduces neighborhood summation at the cost of additional spectrum and workspace storage. We allow at most two algorithm switches along the fine-to-coarse level sequence. Let \(G_\ell\) be the number of GPUs used at level \(\ell\). As the grid is coarsened, subdomains can be aggregated onto fewer GPUs to reduce communication and synchronization costs.

The remaining configuration choices concern the internal precision of the multigrid V-cycle and the coarsest grid size. Let \(n_{64}\) be the number of levels, starting from the finest, that use FP64; the remaining levels use FP32. Thus, \(n_{64}=0\) and \(n_{64}=J+1\) correspond to all-FP32 and all-FP64 internal computations, respectively. This choice balances memory use against preconditioner effectiveness~\cite{HighamMary2022MixedPrecision}. Let \(N_J\) denote the number of intervals in each coordinate direction on the coarsest grid. It determines the number of grid levels and the coarsest-level workload. A candidate configuration is written as
\begin{equation}\label{eq:hierarchical-candidate}
c=(\bm B,\bm G,n_{64},N_J),
\qquad
\bm B=(B_0,\ldots,B_J),\quad
\bm G=(G_0,\ldots,G_J).
\end{equation}

A candidate must support successive factor-two coarsening down to the chosen coarsest grid and satisfy \(G_{\ell+1}\le G_\ell\), with integer ratios \(G_\ell/G_{\ell+1}\). On levels requiring inter-GPU exchange, the local plane count must exceed \(2r_{x,\ell}\), where \(r_{x,\ell}\) is the half-bandwidth along \(x\). This leaves an interior region whose computation can overlap communication. Each configuration must also meet the per-GPU memory budget and the prescribed convergence requirements. Under these constraints, we use total solution time \(T_{\mathrm{solve}}(c)\) and maximum per-GPU memory use \(M_{\mathrm{peak}}(c)\) to select time- and capacity-oriented configurations. The next subsection describes the agent-assisted screening procedure.

\begin{figure}[!htb]
\centering
\includegraphics[width=\textwidth]{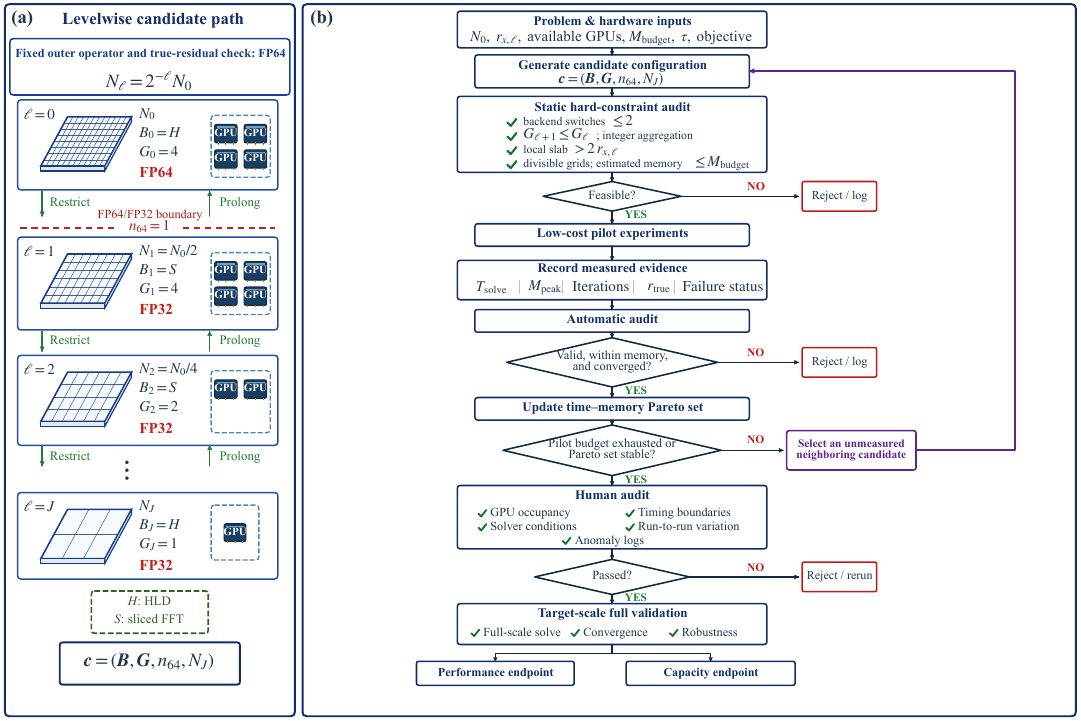}
\caption{MG-Pilot framework for joint configuration across multigrid levels. (a) Levelwise algorithm, GPU allocation, and precision settings, with the coarsest grid size. (b) Feasibility checks, pilot experiments, and audits, followed by target-scale validation to select time- and capacity-oriented configurations.}
\label{fig:agent-framework}
\end{figure}
\subsection{Agent-assisted configuration screening}
\label{sec:hierarchical-configuration-agent}

MG-Pilot generates, tests, and screens candidate configurations within a prescribed experimental budget. As illustrated in Figure~\ref{fig:agent-framework}, feasibility checks and pilot experiments are followed by human audits and validation at the target scale. Results from the pilot experiments guide the choice of subsequent tests.

First, MG-Pilot excludes configurations that violate the grid and resource constraints in Section~\ref{sec:configuration-space}. It then conducts low-cost pilot experiments on feasible candidates, prioritizing endpoint choices and representative intermediate combinations of algorithms, GPU allocations, and precision. The agent generates run configurations, invokes the solver, and extracts solution time \(T_{\mathrm{solve}}(c)\), maximum per-GPU memory use \(M_{\mathrm{peak}}(c)\), and iteration counts from the logs. After each solve, it checks the measured peak memory use against the prescribed memory budget and recomputes the true relative residual \(r_{\mathrm{true}}=\|\bm F_h-A_h\bm U_h\|_2/\|\bm F_h\|_2\) to verify convergence.

After each round of experiments, MG-Pilot compares the solution times and memory use of configurations that satisfy the memory and convergence requirements. It retains non-dominated configurations: a configuration is discarded if another is no slower and uses no more memory, with at least one strict improvement. The agent then builds on the retained configurations, prioritizing tests that vary one configuration choice at a time, and updates the candidates within the experimental budget.

After the pilot experiments, a human audit checks the numerical results and which operations are included in the timings. Configurations that pass the audit undergo complete solves at the target scale. Among the validated configurations at each problem size, those with the shortest solution time and the lowest peak memory use are selected as the time- and capacity-oriented configurations, respectively.

\section{Numerical experiments}\label{sec:numerical}

%

We evaluate the proposed solvers in terms of computational efficiency,
parallel scalability, and memory use, organizing the study around four
questions: (i) how do the three matrix--vector multiplication algorithms
compare on a single GPU, and how does multigrid preconditioning affect them?
(ii) how do they scale on multiple GPUs, and what role does the communication
pattern play? (iii) can agent-assisted joint configuration improve solution
time and memory use across multigrid levels? (iv) how do the solvers perform
on large-scale transient nonlocal diffusion and phase-field problems?

Section~\ref{sec:single-baseline} addresses (i) through iteration counts,
solution times, operator throughput, and Roofline measurements.
Section~\ref{sec:multi-scaling} addresses (ii) through strong- and
weak-scaling tests on up to eight GPUs.
Section~\ref{sec:agent-optimization} addresses (iii) through time- and
capacity-oriented configurations selected by the agent.
Section~\ref{sec:transient-applications} addresses (iv) through an insulated
penny-shaped cutout and large-scale star-shaped interface evolution.

Single- and multi-GPU performance tests use NVIDIA GeForce RTX~4090 GPUs with
24~GiB of memory per device, with up to eight GPUs used in the multi-GPU
experiments. The transient diffusion experiments additionally include runs on
four NVIDIA A800 GPUs. The CPU baseline for the single-GPU comparisons uses an
Intel Core i9-9940X with 14 cores and 28 threads. The GPU implementation runs
on Linux 6.8 and uses Python 3.10.12, CuPy 13.3.0, and CUDA Runtime 12.6.

\subsection{Single-GPU solution and algorithm performance}
\label{sec:single-baseline}
\begin{table}[H]
\centering
\caption{Iterations, solution times, and speedups over CPU--CG for six single-GPU solvers.}
\label{tab:single-gpu-solver-summary}
\small
\begin{tabular*}{0.90\textwidth}{@{\extracolsep{\fill}}rllrrr@{}}
\toprule
\(N\) & Method & Backend/preconditioner & Iter & \(T_{\rm solve}\)(s) & \(S_{\rm CPU}\)\\
\midrule
128 & CG     & direct/none & 60 & 0.2111 & 95.64\\
    & CG     & 3D FFT/none & 57 & 0.0731 & 276.10\\
    & CG     & sliced FFT/none & 57 & 0.0809 & 249.67\\
    & PCG    & direct/MG   & 10 & 0.1762 & 114.62\\
    & PCG    & 3D FFT/MG   & 10 & 0.1890 & 106.85\\
    & PCG    & sliced FFT/MG & 10 & 0.0863 & 233.98\\
\midrule
256 & CG     & direct/none & 115 & 3.4007 & 75.84\\
    & CG     & 3D FFT/none & 114 & 1.1868 & 217.29\\
    & CG     & sliced FFT/none & 114 & 1.0899 & 236.63\\
    & PCG    & direct/MG   & 11 & 1.4164 & 182.08\\
    & PCG    & 3D FFT/MG   & 11 & 0.6015 & 428.71\\
    & PCG    & sliced FFT/MG & 11 & 0.4579 & 563.19\\
\midrule
512 & CG     & direct/none & 230 & 54.877 & 75.33\\
    & CG     & 3D FFT/none & 226 & 26.781 & 154.37\\
    & CG     & sliced FFT/none & 226 & 19.398 & 213.13\\
    & PCG    & direct/MG   & 12 & 12.283 & 336.57\\
    & PCG    & 3D FFT/MG   & 12 & 5.997 & 689.34\\
    & PCG    & sliced FFT/MG & 12 & \textbf{4.054} & \textbf{1,019.78}\\
\bottomrule
\end{tabular*}
\end{table}
The single-GPU experiments evaluate how the matrix--vector multiplication algorithm and multigrid preconditioning affect solution efficiency. We compare iteration counts and total solution times, and examine operator throughput and Roofline measurements to explain the performance differences. For these tests, we solve~\eqref{eq:continuous-problem} on \(\Omega=(0,1)^3\). Each coordinate direction is divided into \(N\) intervals, with \(h=1/N\) and \(\delta=5h\). Let \(p(s)=s(1-s)\), and choose the exact solution \(u_{\rm ex}(x,y,z)=p(x)p(y)p(z)\). We impose the volume constraint \(u=u_{\rm ex}\) on \(\Omega_c\) and determine the right-hand side by substituting \(u_{\rm ex}\) into~\eqref{eq:continuous-problem}. The kernel is \(\rho_\delta(r)=3/(\pi\delta^4r)\) for \(0<r<\delta\), which satisfies the second-moment condition in Section~\ref{sec:acfd-discretization}. All computations use FP64, and iterations stop when the relative residual falls below \(\tau=10^{-8}\).

We compare CG and MG--PCG using direct evaluation, full 3D FFT, and sliced FFT. Table~\ref{tab:single-gpu-solver-summary} reports iteration counts, solution times, and speedups relative to 28-thread CPU--CG using the same grid-cell ACFD discretization and grid size.

As the grid is refined, the CG iteration counts increase substantially, whereas MG--PCG requires only 10--12 iterations over the tested grid sizes. Fewer iterations do not always yield shorter solution times: at \(N=128\), both FFT-based MG--PCG solvers are slower than their CG counterparts. At \(N=512\), multigrid reduces the iteration count for direct evaluation from 230 to 12 and the solution time from 54.877 to 12.283~s. Using sliced FFT further reduces the MG--PCG solution time to 4.054~s, with the iteration count unchanged at 12. This configuration achieves a speedup of 1,019.78 over CPU--CG. Multigrid improves conditioning and reduces iteration counts, while FFT-based reformulation reduces the cost of each matrix--vector product. These effects complement each other on larger grids.

To explain these performance differences, we examine the throughput and floating-point performance of individual matrix--vector products. Throughput is defined as the number of degrees of freedom processed per second. We use CUDA events to time matrix--vector products separately for each algorithm. Floating-point operations and DRAM traffic are measured with Nsight Compute. The Roofline model~\cite{Williams2009Roofline} uses the measured STREAM bandwidth of \(930.2\ {\rm GB\,s^{-1}}\) and FP64 peak performance of \(1.140\ {\rm TFLOP\,s^{-1}}\) on the same device.

\begin{figure}[!htbp]
\centering
\begin{minipage}[t]{0.48\textwidth}
\centering
\subfigure[Degree-of-freedom throughput per matrix--vector product]{\includegraphics[width=0.8\linewidth]{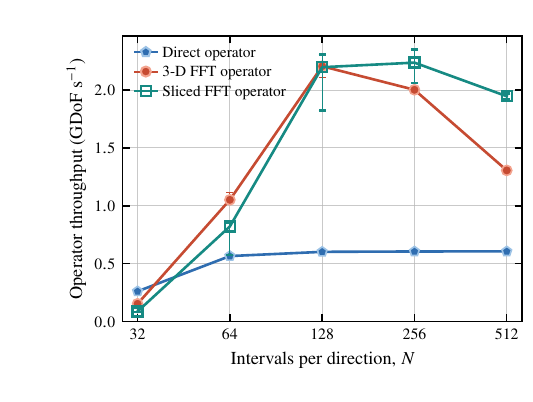}\label{fig:gpu-operator-throughput}}
\end{minipage}\hfill
\begin{minipage}[t]{0.48\textwidth}
\centering
\subfigure[FP64 Roofline based on measured memory traffic]{\includegraphics[width=0.8\linewidth]{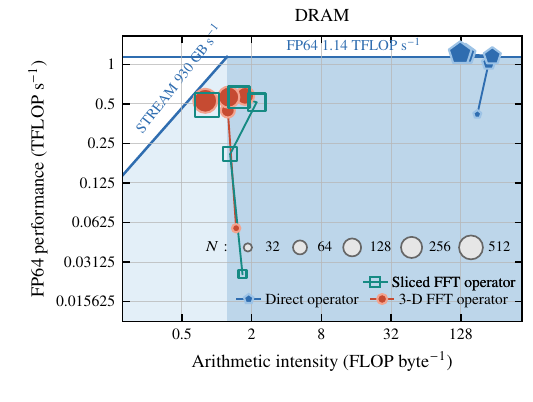}\label{fig:gpu-operator-roofline}}
\end{minipage}
\caption{Single-GPU performance of the three algorithms. (a) Degree-of-freedom throughput of direct evaluation, full 3D FFT, and sliced FFT; error bars show the range of repeated timings. (b) FP64 Roofline measurements, with increasing marker size corresponding to \(N=32,64,128,256,512\).}
\label{fig:gpu-operator-performance}
\end{figure}

Figure~\ref{fig:gpu-operator-performance}(a) compares the throughput of direct evaluation (HLD), full 3D FFT, and sliced FFT. HLD throughput levels off at approximately \(0.61\ {\rm GDoF\,s^{-1}}\) for \(N\ge128\). Both FFT algorithms reach peak throughput at intermediate grid sizes and then decline. At \(N=512\), full 3D FFT and sliced FFT achieve 2.15 and 3.21 times the throughput of direct evaluation, respectively. At \(N=512\), sliced FFT achieves \(1.95\ {\rm GDoF\,s^{-1}}\), approximately 13.0\% below its peak. This decline is smaller than the 40.8\% decrease for full 3D FFT, and sliced FFT retains the highest throughput at the largest tested size.

Figure~\ref{fig:gpu-operator-performance}(b) helps explain these differences in terms of arithmetic intensity and floating-point work. HLD exhibits high arithmetic intensity and achieved FP64 throughput at the larger tested grid sizes. It still requires approximately 2,000 neighbor operations per degree of freedom. The FFT algorithms instead benefit from the reduced floating-point work of the transformed matrix representation. At \(N=512\), full 3D FFT and sliced FFT require approximately 404 and 251 FP64 operations per degree of freedom, respectively. Sliced FFT further reduces computational cost by omitting the Fourier transform along \(x\). For both FFT algorithms, arithmetic intensity does not increase monotonically with grid size. At the larger tested grid sizes, reduced floating-point work allows the FFT algorithms to achieve higher operator throughput despite lower achieved FP64 rates.

\subsection{Multi-GPU performance and scalability}
\label{sec:multi-scaling}

Building on the single-GPU comparisons, we evaluate the solution efficiency and scalability of the three operator algorithms on multiple GPUs, with and without multigrid preconditioning. The tests use \(P=1,2,4,8\) GPUs on a system with four GPUs connected to each non-uniform memory access (NUMA) node. Four-GPU runs remain within one NUMA node, whereas eight-GPU runs span two nodes.

Communication is overlapped with computation. HLD asynchronously exchanges physical-space halos on both sides of each subdomain while evaluating local stencils that require no remote data. Sliced FFT overlaps asynchronous spectral-halo exchange with interior spectral-stencil evaluation. Full 3D FFT requires global exchanges involving all participating GPUs, limiting the extent to which communication can be hidden. For strong scaling, we fix the global grid size at \(N=512\) and define the speedup and parallel efficiency as \(S_P=T_1/T_P\) and \(\eta_P=S_P/P\), where \(T_P\) is the total solution time on \(P\) GPUs. Table~\ref{tab:local-operator-scaling} summarizes the results for CG and MG--PCG.

\begin{table}[H]
\centering
\caption{Strong scaling at \(N=512\) for the grid-cell ACFD discretization. For each solver and \(P=1,2,4,8\), \(T_P\) is the shortest solution time among repeated runs on \(P\) GPUs. The first numerical column reports \(T_1\) in seconds; the remaining columns report speedup \(S_P=T_1/T_P\) and parallel efficiency \(\eta_P=S_P/P\) for \(P=2,4,8\), with efficiencies expressed as percentages.}
\label{tab:local-operator-scaling}
\small
\begin{tabular*}{\textwidth}{@{\extracolsep{\fill}}lrrrrrrr@{}}
\toprule
Solver & \(T_1\)(s) & \(S_2\) & \(\eta_2\) & \(S_4\) & \(\eta_4\) & \(S_8\) & \(\eta_8\)\\
\midrule
HLD--CG              & 54.877 & 1.990 & 99.5\% & 3.998 & 99.9\% & 7.781 & 97.3\%\\
Full 3D FFT--CG      & 26.781 & 1.107 & 55.4\% & 1.636 & 40.9\% & 2.259 & 28.2\%\\
Sliced FFT--CG        & 19.398 & 1.970 & 98.5\% & 3.758 & 94.0\% & 7.456 & 93.2\%\\
HLD--MG--PCG          & 12.283 & 1.961 & 98.1\% & 3.768 & 94.2\% & 7.105 & 88.8\%\\
Full 3D FFT--MG--PCG & 5.997 & 0.892 & 44.6\% & 0.914 & 22.9\% & 2.540 & 31.7\%\\
Sliced FFT--MG--PCG   & 4.054  & 1.938 & 96.9\% & 3.352 & 83.8\% & 6.018 & 75.2\%\\
\bottomrule
\end{tabular*}
\end{table}

Although HLD--CG has the longest single-GPU solution time, it achieves an eight-GPU speedup of 7.781, corresponding to a parallel efficiency of 97.3\%. Sliced FFT--CG combines a shorter single-GPU solution time with an eight-GPU speedup of 7.456 and an efficiency of 93.2\%. In contrast, full 3D FFT--CG achieves an eight-GPU speedup of only 2.259, with an efficiency of 28.2\%. These results show that the communication pattern strongly affects how single-GPU computational performance translates into multi-GPU solution efficiency.

Sliced FFT--MG--PCG gives the shortest solution times for all tested GPU counts. On eight GPUs, its speedup is 6.018 and its parallel efficiency is 75.2\%, compared with 93.2\% for sliced FFT--CG. Small-batch transverse FFTs on coarse grids and communication between multigrid levels can limit the parallel efficiency of the preconditioned solver. Full 3D FFT--MG--PCG is particularly sensitive to parallel overhead: its speedups on two and four GPUs are 0.892 and 0.914, respectively, so both runs take longer than the corresponding single-GPU solve. Thus, parallel efficiency must be considered together with total solution time when comparing the solvers.

Figure~\ref{fig:strong-weak-scaling}(a) plots the CG speedups from Table~\ref{tab:local-operator-scaling} against the ideal strong-scaling curve \(S_P=P\). Figure~\ref{fig:strong-weak-scaling}(b) examines weak scaling of a single matrix--vector product. Each GPU holds \(255^3\) interior nodes, and the global domain is extended along \(x\) as the GPU count increases. With a fixed workload per GPU, the matrix--vector multiplication time should ideally remain constant, giving a normalized time of one.

\begin{figure}[H]
\centering
\includegraphics[width=0.8\textwidth]{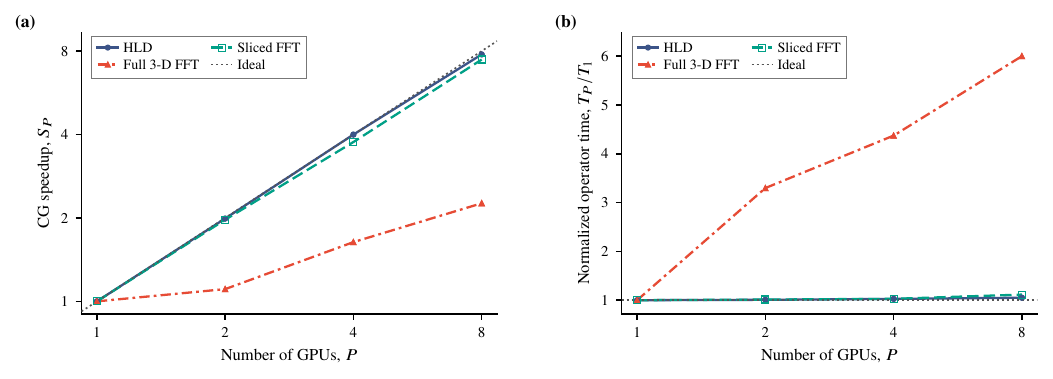}
\caption{Multi-GPU scalability of the three algorithms. (a) CG solution speedups at \(N=512\) from Table~\ref{tab:local-operator-scaling}; the dotted line shows the ideal speedup \(S_P=P\). (b) Weak scaling of a single matrix--vector product with \(255^3\) interior nodes per GPU; the dotted line shows the ideal normalized time \(T_P/T_1=1\).}
\label{fig:strong-weak-scaling}
\end{figure}

HLD and sliced FFT remain close to the ideal behavior in both tests. On eight GPUs, their normalized matrix--vector multiplication times are approximately 1.051 and 1.115, respectively, whereas that of full 3D FFT increases to approximately 6.001. The eight-GPU runs also involve communication between the two NUMA nodes, which can contribute to departures from ideal scaling. Taken together, the strong- and weak-scaling results show that neighboring halo exchanges allow HLD and sliced FFT to retain good scalability, whereas global data redistribution limits the scalability of full 3D FFT.

\subsection{Evaluation of agent-assisted joint optimization}
\label{sec:agent-optimization}


The computational and memory trade-offs between HLD and sliced FFT motivate
joint configuration across multigrid levels. The configuration space across
multigrid levels is large. In the six-level hierarchy considered here, each
level can, in principle, combine two matrix--vector multiplication
algorithms, up to three GPU counts (four, two, or one), and two arithmetic
precisions (FP64 and FP32) for the V-cycle, while the outer PCG remains in
FP64. Before imposing cross-level and resource constraints, these choices
alone yield approximately \((2\times3\times2)^6\approx3.0\times10^6\)
combinations, without accounting for the choice of coarsest-grid size.
Through feasibility screening, MG-Pilot retained 22 candidate configurations.
Further screening and human auditing yielded nine representative
agent-selected configurations, which together with the all-FP64 baseline are
reported in Table~\ref{tab:mixed-capacity}. By comparison, an exhaustive
manual search would require several weeks of developer effort, since each
candidate requires changes to the CuPy/CUDA execution paths, recompilation,
testing, and performance profiling.

We first evaluate the time- and capacity-oriented configurations selected by
MG-Pilot on four GPUs and assess the effect of reduced V-cycle precision on
solution accuracy. We then examine large-scale solution performance and
attainable problem sizes on eight GPUs.

\subsubsection{Joint optimization for large-scale solution}

We use the manufactured-solution problem from Section~\ref{sec:single-baseline} at \(N=1{,}024\) on four GPUs. FP64 sliced FFT exceeds the available GPU memory on the finest grid, so we use an all-FP64 HLD--MG--PCG solver as the baseline. The baseline keeps four GPUs active at every level and uses a coarsest grid of \(N_J=128\). For the candidate configurations, HLD is fixed for the outer matrix--vector products and the finest V-cycle level. MG-Pilot selects HLD or sliced FFT for the remaining levels, together with GPU allocations and internal precision, while the outer PCG remains in FP64. The agent-selected candidates shown in Table~\ref{tab:mixed-capacity} use \(N_J=32\), with the coarsest grid aggregated onto one GPU. Here, \(S_{\mathrm{cfg}}\) denotes the speedup relative to the baseline.

\begin{table}[H]
\centering
\caption{Four-GPU HLD--MG--PCG baseline and agent-selected joint configurations at \(N=1{,}024\) for the grid-cell ACFD discretization. \(H\) and \(S\) denote HLD and sliced FFT, respectively; paths are ordered from fine to coarse levels, and the outer PCG remains in FP64.}
\label{tab:mixed-capacity}
\footnotesize
\setlength{\tabcolsep}{1.9pt}
\begin{tabular*}{\textwidth}{@{\extracolsep{\fill}}clllrrr@{}}
\toprule
Configuration & Backend path & GPU path & FP32 levels
& \(T_{\rm solve}\)(s) & \(M_{\rm peak}\)(GiB) & \(S_{\rm cfg}\)\\
\midrule
HLD--MG--PCG & \(H-H-H-H\) & \(4-4-4-4\) & -- & 46.349 & 13.084 & 1.00\\
\(C_{\rm perf}\) & \(H-S-S-S-H-H\) & \(4-4-4-2-1-1\) & 1,024--32 & \textbf{20.060} & 15.090 & \textbf{2.31}\\
\(C_{\rm mem}\) & \(H-H-H-H-H-H\) & \(4-4-4-4-2-1\) & 512--32 & 26.269 & \textbf{12.770} & 1.76\\
\(C_2\) & \(H-H-H-H-H-H\) & \(4-4-4-4-2-1\) & 256--32 & 27.257 & 13.047 & 1.70\\
\(C_3\) & \(H-H-H-H-H-H\) & \(4-4-4-2-1-1\) & 256--32 & 27.254 & 13.059 & 1.70\\
\(C_4\) & \(H-S-S-S-H-H\) & \(4-4-4-2-1-1\) & 512--32 & 25.226 & 14.219 & 1.84\\
\(C_5\) & \(H-H-S-S-S-H\) & \(4-4-4-4-2-1\) & 256--32 & 27.284 & 13.236 & 1.70\\
\(C_6\) & \(H-H-S-S-S-H\) & \(4-4-4-4-1-1\) & 256--32 & 27.136 & 13.244 & 1.71\\
\(C_7\) & \(H-H-H-S-S-H\) & \(4-4-4-2-1-1\) & 256--32 & 27.223 & 13.119 & 1.70\\
\(C_8\) & \(H-H-S-S-H-H\) & \(4-4-4-2-1-1\) & 256--32 & 27.086 & 13.252 & 1.71\\
\bottomrule
\end{tabular*}
\end{table}

\begin{table}[H]
\centering
\caption{Errors relative to the exact solution and observed convergence orders for all-FP64 and mixed-precision solves. In the mixed-precision runs, the V-cycle uses FP32, while the outer PCG and its matrix--vector products use FP64.}
\label{tab:mixed-accuracy}
\small
\begin{tabular*}{0.90\textwidth}{@{\extracolsep{\fill}}clrrr@{}}
\toprule
\(N\) & Precision & \(E_\infty\) & Order & \(\Delta_E\)\\
\midrule
128 & FP64 & \(9.389693\times10^{-7}\) & -- & --\\
 & mixed & \(9.389693\times10^{-7}\) & -- & \(5.66\times10^{-8}\)\\
\addlinespace[2pt]
256 & FP64 & \(2.326761\times10^{-7}\) & 2.013 & --\\
 & mixed & \(2.326761\times10^{-7}\) & 2.013 & \(1.41\times10^{-7}\)\\
\addlinespace[2pt]
512 & FP64 & \(5.796708\times10^{-8}\) & 2.005 & --\\
 & mixed & \(5.796708\times10^{-8}\) & 2.005 & \(2.94\times10^{-8}\)\\
\addlinespace[2pt]
1,024 & FP64 & \(1.450308\times10^{-8}\) & 1.999 & --\\
 & mixed & \(1.450307\times10^{-8}\) & 1.999 & \(2.06\times10^{-7}\)\\
\bottomrule
\end{tabular*}
\end{table}
Among the configurations in Table~\ref{tab:mixed-capacity}, \(C_{\mathrm{perf}}\) is fastest at 20.060~s, achieving a speedup of 2.31 over the baseline. Its peak memory use is 15.090~GiB, compared with 13.084~GiB for the baseline. \(C_{\mathrm{mem}}\) uses the least memory, 12.770~GiB, and solves the problem in 26.269~s. Configurations \(C_4\) and \(C_{\mathrm{perf}}\) share the same algorithm and GPU paths and differ only in the precision of the finest V-cycle level. Switching this level from FP64 to FP32 reduces solution time from 25.226 to 20.060~s, a reduction of 20.5\%.

To assess the effect of reduced internal precision on solution accuracy, Table~\ref{tab:mixed-accuracy} compares all-FP64 computations with a mixed-precision variant using FP32 internally and FP64 in outer PCG. The approximate-ball geometry is refined as the grid spacing is halved, keeping geometric and discretization errors of the same order. The final relative residuals in this test are all below \(10^{-10}\). We measure the solution error by \(E_\infty=\max_{\bm i\in\mathcal G_h}|U_{\bm i}-u_{\rm ex}(\bm x_{\bm i})|\). Denote the errors by \(E_\infty^{64}\) and \(E_\infty^{\rm mix}\), and measure their relative difference by \(\Delta_E=\lvert E_\infty^{\rm mix}-E_\infty^{64}\rvert/E_\infty^{64}\). The observed order is \(p=\log_2(E_\infty(N)/E_\infty(2N))\).

Both precision choices exhibit second-order convergence, and the maximum relative difference between their solution errors is \(2.06\times10^{-7}\). These results show that using FP32 within the V-cycle retains solution accuracy comparable to all-FP64 computations over the tested grids.

\subsubsection{Large-scale computation on eight GPUs}

We assess large-scale solution performance and capacity on eight RTX~4090 GPUs using the manufactured-solution problem from Section~\ref{sec:single-baseline}. Field variables and discrete weights are stored in FP32 to reduce memory use, and V-cycle computations use FP32. Neighborhood accumulation and diagonal calculations in the outer HLD algorithm, together with inner-product reductions, use FP64. For the performance tests, HLD is fixed for the outer matrix--vector products and the finest V-cycle level, while sliced FFT is used on selected coarse levels. GPU aggregation paths and coarsest grids are compared. The first four rows of Table~\ref{tab:eight-gpu-large-scale} report the fastest tested configurations at \(N=1{,}536,1{,}728,1{,}920,2{,}048\), while the last two rows report capacity tests. Degrees of freedom are counted as the \((N-1)^3\) interior nodes.

\begin{table}[H]
\centering
\caption{Eight-GPU large-scale results. \(H\) and \(S\) denote HLD and sliced FFT; paths are ordered from fine to coarse grids. The last two rows report capacity tests.}
\label{tab:eight-gpu-large-scale}
\small
\begin{tabular*}{\textwidth}{@{\extracolsep{\fill}}rrccrr@{}}
\toprule
\(N\) & DoFs & Backend path & GPU path & \(N_J\) & \(T_{\rm solve}\) (s)\\
\midrule
1,536 & 3,616,805,375 & \(H-S-S-S-S-S-H\) & \(8-8-8-8-8-4-1\) & 24 & 43.018\\
1,728 & 5,150,827,583 & \(H-S-S-S-S-S-H\) & \(8-8-8-8-8-4-1\) & 27 & 61.377\\
1,920 & 7,066,834,559 & \(H-S-S-S-S-S-H\) & \(8-8-8-4-2-1-1\) & 30 & 92.757\\
2,048 & 8,577,357,823 & \(H-H-H-S-S-S-H\) & \(8-8-8-8-8-4-1\) & 32 & 115.649\\
\midrule
2,112 & 9,407,293,631 & \(H-H-H-H-H-H-H\) & \(8-8-8-8-4-2-1\) & 33 & 151.412\\
2,176 & 10,289,109,375 & \(H-H-H-H-H-H-H\) & \(8-8-8-8-4-2-1\) & 34 & 166.071\\
\bottomrule
\end{tabular*}
\end{table}

At the first three problem sizes, the selected configurations use sliced FFT from the first coarse level. At \(N=2{,}048\), memory constraints require sliced FFT to start at the \(N=256\) level, with HLD used on the finer levels.

The capacity tests use all-HLD configurations at \(N=2{,}112\) and \(N=2{,}176\). The largest tested grid, \(N=2{,}176\), contains approximately 10.29 billion degrees of freedom and is solved in 166.071~s.

\subsection{Applications to 3D transient nonlocal diffusion and phase-field evolution}
\label{sec:transient-applications}

We next apply the proposed solvers to 3D  transient nonlocal diffusion and nonlinear phase-field evolution. The first example involves an insulated penny-shaped cutout and compares temperature evolution and simulation cost with published results. The second demonstrates large-scale star-shaped interface evolution and convergence toward the local limit.

\subsubsection{Transient diffusion with an insulated penny-shaped cutout}
\label{sec:insulated-cutout}

We consider the insulated penny-shaped cutout problem studied in~\cite{Jafarzadeh2021FCBM}. The temperature satisfies
\begin{equation}
\partial_t u(\bm x,t)-D\mathcal L_\delta u(\bm x,t)
=f(\bm x,t),\qquad \bm x\in\Omega,\ t>0.
\label{eq:cutout-problem}
\end{equation}
Here, \(D\) is the diffusion coefficient and \(f\) is the volumetric heat source. The physical domain is \(\Omega=(-1,1)^3\setminus\mathcal C\), where \(\mathcal C=\{\bm x:x_1^2+x_2^2<0.7^2,\ |x_3|<0.1\}\) is the central cutout. Denote the upper and lower cube faces by \(\Gamma_D^\pm=\{\bm x\in\partial(-1,1)^3:x_3=\pm1\}\), and the other four faces by \(\Gamma_{\rm side}\). We set \(D=1\), \(f=0\), and zero initial temperature, with \(u=1\) on \(\Gamma_D^+\) and \(u=0\) on \(\Gamma_D^-\). Insulation on \(\Gamma_{\rm side}\) and the upper and lower cutout faces is imposed by reflection extensions following~\cite{Jafarzadeh2021FCBM}.

A uniform grid is defined on the fixed cube \(B_{L_b}=(-L_b/2,L_b/2)^3\), with \(L_b=2.56\) and \(N\) intervals in each direction. We set \(h=L_b/N\), \(\delta=5h\), and \(k=3\). The volume-constraint layer used for boundary extension has thickness \(\delta\), which decreases under grid refinement. The gap between corresponding faces of the computational box and physical cube remains \((L_b-2)/2=0.28\). It contains the volume-constraint layer and additional space, as shown in Figure~\ref{fig:cutout-geometry}. The kernel is \(\rho_\delta(r)=9/(4\pi\delta^3r^2)\) for \(0<r<\delta\). We use \(\Delta t=10^{-3}\) and advance 1,000 steps to \(T=1\). On four RTX~4090 GPUs at \(N=1{,}024\), we use configuration \(C_{\rm perf}\).
\begin{figure}[H]
\centering
\includegraphics[width=0.6\textwidth]{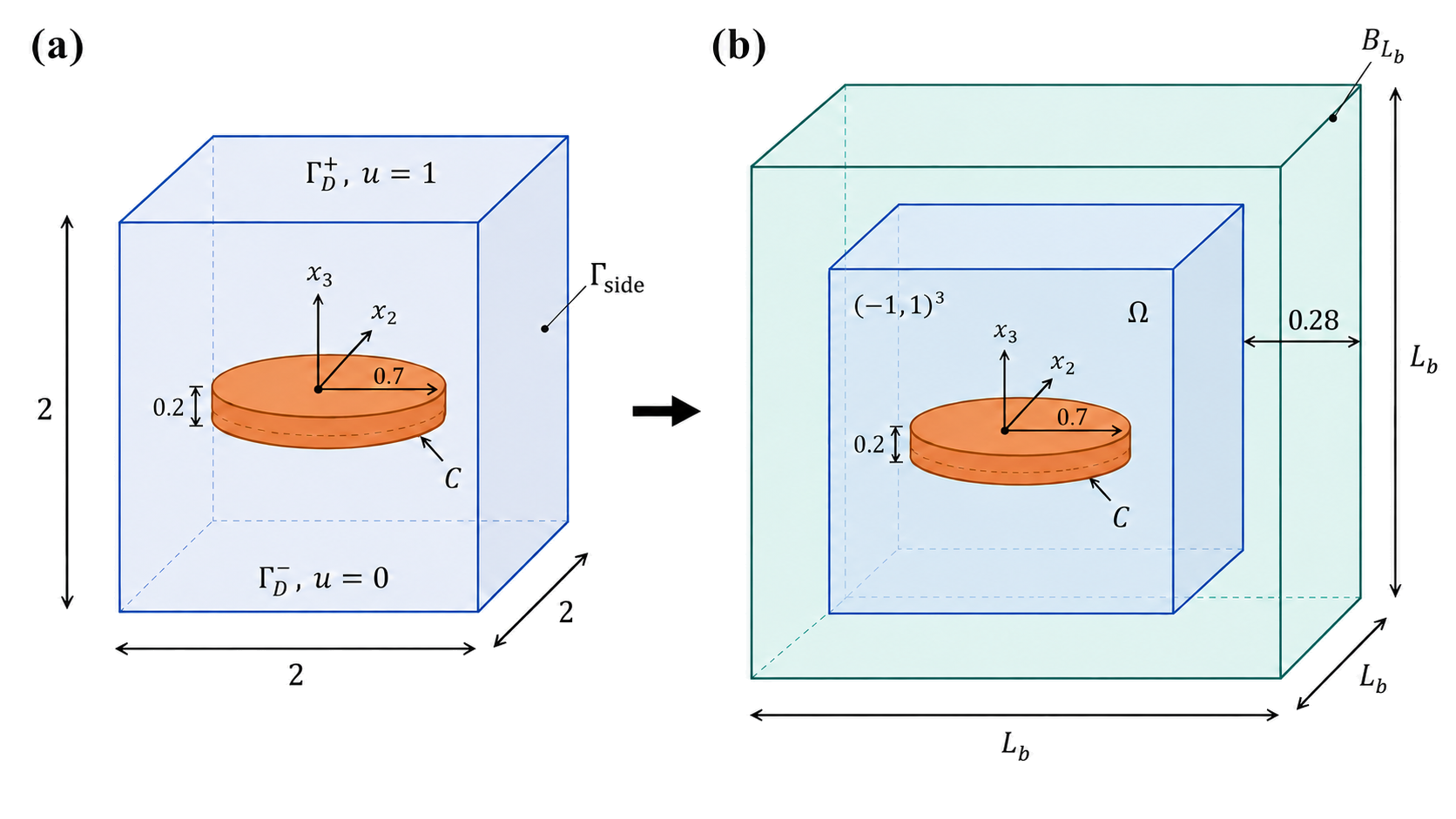}
\caption{Geometry of the insulated penny-shaped cutout problem. (a) Physical domain \(\Omega\), central cutout \(\mathcal C\), and thermal boundary conditions. (b) Physical domain within the fixed computational box \(B_{L_b}\).}
\label{fig:cutout-geometry}
\end{figure}
\begin{figure}[H]
\centering
\includegraphics[width=0.7\textwidth]{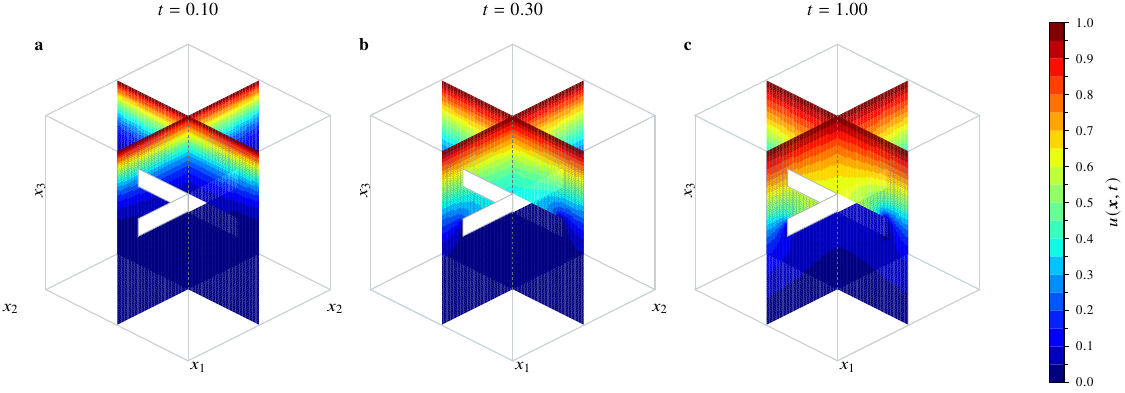}
\caption{Transient 3D nonlocal diffusion with an insulated penny-shaped cutout (grid-cell ACFD, \(N=512\)). Panels (a)--(c) correspond to \(t=0.1,0.3,1.0\), respectively. Each panel shows two perpendicular vertical temperature sections, \(x_1=0\) and \(x_2=0\). The white V-shaped region marks the intersection of the removed cutout with these sections, not zero temperature. All panels use the same color scale.}
\label{fig:cutout-evolution}
\end{figure}
Let \(t_n=n\Delta t\), let \(\bm U_h^n\) denote the grid solution, and let \(\bm F_h(t_n)\) include the source and volume-constraint contributions. Backward Euler discretization gives
\begin{equation}
\frac{\bm U_h^{n+1}-\bm U_h^n}{\Delta t}
+D A_h\bm U_h^{n+1}=\bm F_h(t_{n+1}).
\label{eq:transient-be}
\end{equation}
Each time step is solved by MG--PCG. The preconditioner is constructed from the shifted operators \(I+\Delta t D A_\ell\) and uses the symmetric V-cycle of Section~\hyperlink{subsection.4.1}{4.1}.

Figure~\ref{fig:cutout-evolution} shows the temperature evolution at \(N=512\). Heat diffuses downward from the upper face, with isotherms bending around the cutout and a marked delay in warming below it. The main features agree with~\cite{Jafarzadeh2021FCBM}. Table~\ref{tab:cutout-runtime} compares total runtimes. The reference advances about 18,200 explicit time steps; our backward-Euler solver advances 1,000 implicit steps.

At \(N=512\), four RTX~4090 GPUs require 0.210~h, compared with 3.9~h on 16 CPU cores and 2.1~h on one V100 for the reference method. As the grid grows, memory capacity further affects algorithm selection. At \(N=1{,}024\), four A800 GPUs require 7.710~h with the same joint configuration as the RTX~4090 runs. Replacing both outer PCG matrix--vector products and those on the finest V-cycle level with sliced FFT reduces the runtime to 1.679~h. At \(N=2{,}048\), the HLD-based finest-level configuration on A800 requires more than 30~h. The larger A800 memory allows sliced FFT on the finest level, reducing the runtime to 9.260~h. Thus, multi-GPU performance depends on both algorithm choice and GPU memory capacity, with larger memory enabling the use of more efficient algorithms, such as sliced FFT on the finest level.

\begin{table}[htbp]
\centering
\caption{Total runtime to \(T=1\) for the insulated penny-shaped cutout problem. FCBM--EC data are taken from Table~1 of Jafarzadeh et al.~\cite{Jafarzadeh2021FCBM}; the present timings exclude weight construction and solver initialization.}
\label{tab:cutout-runtime}
\footnotesize
\renewcommand{\arraystretch}{1.15}
\setlength{\tabcolsep}{4pt}
\begin{tabular}{@{}lllllll@{}}
\toprule
Method & Hardware & \(N=128\) & \(N=256\) & \(N=512\) & \(N=1{,}024\) & \(N=2{,}048\)\\
\midrule
FCBM--EC nodes & -- & 2,097,152 & 16,777,216 & 134,217,728 & 1,073,741,824 & --\\
ACFD DoFs & -- & 2,048,383 & 16,581,375 & 133,432,831 & 1,070,599,167 & 8,577,357,823\\
\midrule
FCBM--EC in~\cite{Jafarzadeh2021FCBM} & 16 CPU & 2.7 min & 34.7 min & 3.9 h & 1.3 days & --\\
FCBM--EC in~\cite{Jafarzadeh2021FCBM} & V100 GPU & 29 min & 38.3 min & 2.1 h & OOM & --\\
FFT/HLD--MG--PCG & RTX4090 GPU \(\times 4\) & 1.065 min & 1.971 min & 0.210 h & 5.983 h & OOM\\
FFT/HLD--MG--PCG & RTX4090 GPU \(\times 8\) & 1.378 min & 2.436 min & 0.174 h & 3.005 h & 27.732 h\\
FFT/HLD--MG--PCG & A800 GPU \(\times 4\) & 1.308 min & 2.875 min & 0.197 h & 1.679 h & 9.260 h\\
\bottomrule
\end{tabular}
\end{table}

\subsubsection{3D  star-shaped interface evolution and the local limit}
\label{sec:allen-cahn-star}
We here extend the computation to the nonlocal Allen--Cahn equation with a nonlinear reaction term \cite{DuJuLiQiao2019ETD}:
\begin{equation}\label{eq:allen-cahn-test}
    \partial_tu=\varepsilon^2\mathcal L_\delta u+u-u^3,
    \qquad u(x,0)=u_0(x).
\end{equation}
Following the 3D  star-shaped interface test in~\cite{Geng2025NonlocalPhasefield}, we impose periodic boundary conditions on \(\Omega=[-1,1)^3\) and set \(\varepsilon=0.06\). The initial condition is
\(u_0(x)=\tanh((R(x)-r)/0.035)\), where \(r=|x|\) and \(R(x)=0.23+0.45(x_1^4+x_2^4+x_3^4)/r^4\); the quotient is set to zero at the origin. The cited 3D  test uses a \(64^3\) grid. Here, we compute up to \(2048^3\) and compare with the local limit.

We use \(N=64,128,256,512,1024,2048\) periodic grid points in each direction, with \(h=2/N\) and \(\delta=5h\). The kernel is \(\rho_\delta(r)=3/(\pi\delta^4r)\) for \(0<r<\delta\) and zero outside the ball. The nonlocal diffusion operator is discretized by the ACFD scheme, and the computation uses its sliced FFT representation.

Time integration uses the second-order exponential time differencing Runge--Kutta scheme (ETDRK2) of Du et al.~\cite{DuJuLiQiao2019ETD}, with stabilization parameter 2 and \(\Delta t=0.005\) up to \(T=10\). For comparison, we replace \(\mathcal L_\delta\) by \(\Delta\) and compute a local reference solution \(U_{\rm loc}\) on a \(1024^3\) grid using Fourier spectral discretization and \(\Delta t=0.0025\). We compare the final states at common nodes using \(E_2=\|U-U_{\rm loc}\|_2/\|U_{\rm loc}\|_2\) and record the discrete energy \(E_h=\varepsilon^2h^3U^TA_hU/2+h^3\sum_i(U_i^2-1)^2/4\).

\begin{figure}[H]
    \centering
    \includegraphics[width=0.84\textwidth]{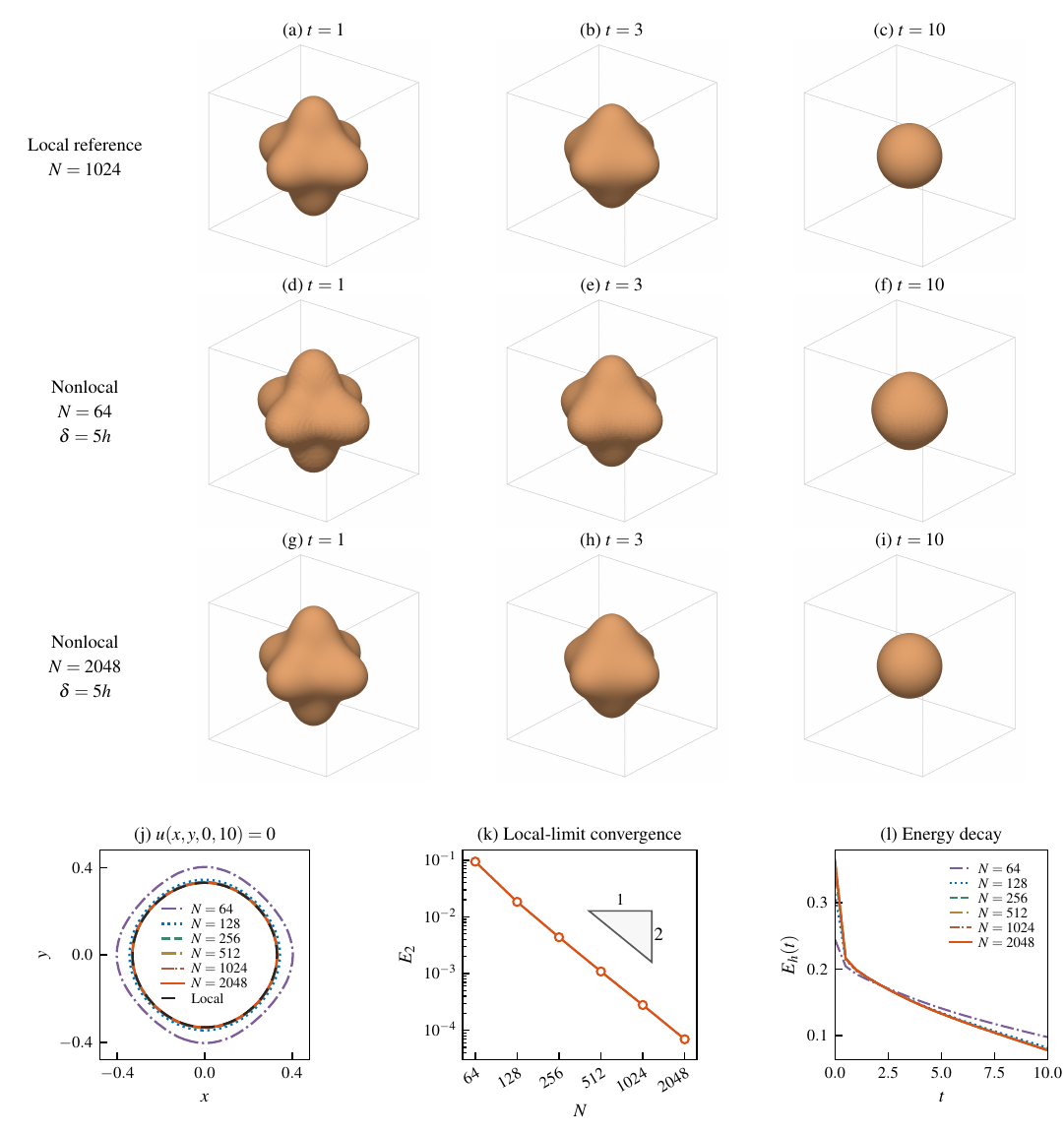}
    \caption{3D  star-shaped interface evolution and the local limit. Panels (a--c) show the local reference solution (\(N=1024\)); (d--f), the nonlocal solution with \(N=64\); and (g--i), the nonlocal solution with \(N=2048\). The \(u=0\) isosurfaces are shown at \(t=1,3,10\). The nonlocal computations use \(\delta=5h\). Quantitative comparisons for \(N=64,128,256,512,1024,2048\): (j) final-state \(u=0\) contours in the \(z=0\) plane; (k) final-state full-field error \(E_2\) relative to the local reference, with a triangle indicating a second-order slope; (l) discrete energy.}
    \label{fig:star-ac}
\end{figure}
\clearpage

In Figure~\ref{fig:star-ac}, the protrusions gradually disappear as the interface shrinks toward a sphere. As the grid is refined from \(N=64\) to \(N=2048\), the interaction radius \(\delta=5h\) decreases accordingly, and the nonlocal numerical solution approaches the local reference solution. The final-state relative error \(E_2\) has convergence rates of 2.35, 2.07, 2.01, 1.96, and 2.01 on successive grids, maintaining approximately second-order AC convergence. The energy decreases with time in all runs. The \(1024^3\) and \(2048^3\) grids contain approximately 1.07 and 8.59 billion nodes, respectively. At \(N=2048\), eight RTX~4090 GPUs complete 2,000 time steps in 2247.0~s, averaging 1.12~s per step.

\section{Conclusions}

We have formulated and numerically validated a sliced FFT algorithm for the
ACFD discretization of three-dimensional peridynamic transient diffusion
problems. The algorithm rests on the following operator identity: the discrete
convolution induced by the ACFD scheme admits a multilevel Toeplitz
representation whose application can be decomposed along slices, so that the
global data redistribution characteristic of full three-dimensional FFTs is
replaced by local communication while the discrete operator itself remains
unaltered. 

The computational complexity of the sliced FFT operator evaluation is
\(\mathcal{O}(N_{\mathrm{dof}}\log N_{\mathrm{dof}})\). The matrix-free
formulation eliminates explicit storage of the global stiffness matrix, and the
multigrid preconditioner reduces the iteration count of the conjugate gradient
solver. MG-Pilot provides a level-wise selection of algorithm, GPU count,
arithmetic precision, and coarsening strategy under human auditing, yielding
configurations optimized for either time-to-solution or memory capacity.

The numerical evidence establishes the following. Strong scaling of both
HLD--CG and sliced FFT--CG is approximately linear on up to eight
RTX~4090 GPUs. The framework accommodates approximately
\(1.029 \times 10^{10}\) degrees of freedom on eight such GPUs. For the
insulated-cutout benchmark, the attainable problem size exceeds that of the
existing fast convolution method, and the computational time is reduced. For the
three-dimensional star-shaped interface evolution problem, the nonlinear
phase-field computation is efficient, and the discrete solution converges to
the local limit as the mesh spacing and interaction radius are reduced
simultaneously, consistent with the asymptotic compatibility of the underlying
discretization.

The sliced FFT evaluation requires only that the discrete operator admit a
multilevel Toeplitz representation with finite interaction range; it is
therefore not restricted to the peridynamic setting. In particular, when a
classical local model is discretized in a form that retains such a
structure, the same operator evaluation strategy applies without
modification, and the asymptotic-compatibility argument carries over. Whether
the resulting solver is competitive with sparse direct or geometric multigrid
methods for purely local operators is a separate question that depends on the
interaction range relative to the grid spacing, and is not claimed here.

Future work will investigate fast solvers for nonuniform and adaptively refined grids. We will also extend the framework to Helmholtz and peridynamic models for elasticity problems and develop suitable preconditioners and multi-GPU solution algorithms.


\end{document}